%% file: quicksort-paths.tex
\documentclass[reqno]{amsart}
\pdfoutput = 1
\newif\iffinal
\finaltrue
\iffinal\else
\usepackage[eso-foot]{svninfo}
\fi

\usepackage[utf8x]{inputenc}
\usepackage[T1]{fontenc}
\usepackage{lmodern}
\usepackage[english]{babel}
\usepackage[draft=false,kerning=true]{microtype}
\usepackage{units}

\PrerenderUnicode{ü}

\usepackage{amsmath, amsfonts, amssymb, amsthm, amstext, bbm}

\usepackage{tikz}
\usetikzlibrary{calc,shapes}
\usepackage{pgfplots}
\usepgfplotslibrary{statistics} 

\usepackage{ifthen,calc}
\usepackage{suffix}
\usepackage{enumerate}
\usepackage{mathtools}
\usepackage{nameref}
\usepackage{hyperref}

\usepackage{algorithm}
\usepackage[noend]{algpseudocode}
\usepackage{accents}

\newcommand{\parentheses}[4][]%
{\mathopen{}\ifthenelse{\equal{#1}{}}{\left#2}{\csname#1\endcsname#2}%
    {#4}\mathclose{}\ifthenelse{\equal{#1}{}}{\right#3}{\csname#1\endcsname#3}}
\newcommand{\foperator}[1]{\mathop{{#1}\empty{}}}
\newcommand{\f}[3][]{\ensuremath{\foperator{#2}\parentheses[#1]{(}{)}{#3}}}

\makeatletter
\newcommand\undisp[1]{\bgroup\@displayfalse #1\egroup}
\makeatother

\newcommand{\iverson}[1]{\ensuremath{\parentheses{[}{]}{#1}}}
\WithSuffix\newcommand{\iverson}*[1]{\ensuremath{\iverson{\text{\normalfont #1}}}}
\newcommand{\abs}[2][]{\ensuremath{%
    \parentheses[#1]{\lvert}{\rvert}{#2}}}
\newcommand{\floor}[1]{\ensuremath{\parentheses{\lfloor}{\rfloor}{#1}}}
\newcommand{\ceil}[1]{\ensuremath{\parentheses{\lceil}{\rceil}{#1}}}

\newcommand{\Ohsymbol}{O}
\newcommand{\Oh}[2][]{\ensuremath{\f[#1]{\Ohsymbol}{#2}}}
\newcommand{\set}[1]{\ensuremath{\parentheses{\{}{\}}{#1}}}
\WithSuffix\newcommand{\set}*[2]{\ensuremath{%
    \left\{{#1}\thinspace\setmiddlesymbol\thinspace{#2}\right\}}}

\newcommand{\calC}{\mathcal{C}}
\newcommand{\Upstep}{\nearrow}
\newcommand{\CUpstep}{\{\Upstep\}}
\newcommand{\Downstep}{\searrow}
\newcommand{\CDownstep}{\{\Downstep\}}
\newcommand{\CDyckpath}{\accentset{\bigtriangleup}{\mathcal{D}}}
\newcommand{\CReflectedDyckpath}{\underaccent{\bigtriangledown}{\mathcal{D}}}
\newcommand{\CZero}{\{\bullet\}}
\newcommand{\Ctimes}{\times}
\newcommand{\CUpblock}{\Upstep\!\CDyckpath\!\Downstep}
\newcommand{\CDownblock}{\Downstep\!\CReflectedDyckpath\!\!\Upstep}
\newcommand{\CTerminalUpblock}{\Upstep\!\CDyckpath}
\newcommand{\E}[2][]{\f[#1]{\mathbb{E}}{#2}}
\renewcommand{\P}[2][]{\f[#1]{\mathbb{P}}{#2}}

\DeclareMathOperator{\artanh}{artanh}

\newcommand{\Ccv}{C^{\mathrm{cv}}}

\newcommand{\Pcv}{P^{\mathrm{cv}}}
\newcommand{\Pct}{P^{\mathrm{ct}}}

\newcommand{\N}{\ensuremath{\mathbbm{N}}}
\newcommand{\Z}{\ensuremath{\mathbbm{Z}}}

\newcommand{\eps}{\ensuremath{\varepsilon}}
\newcommand{\Hodd}{H^{\mathrm{odd}}}
\newcommand{\Halt}{H^{\mathrm{alt}}}

\newcommand{\dd}{\mathrm{d}}
\renewcommand{\MR}[1]{}

\iffinal
\newcommand{\TODO}[1]{}
\else
\newcommand{\TODO}[1]
{\par\fbox{\begin{minipage}{0.9\linewidth}\textbf{TODO:} #1\end{minipage}}\par}
\fi

\newtheorem{thms}{Thm}[section] 
\theoremstyle{plain}
\newtheorem*{theorem*}{Theorem}
\newtheorem{theorem}[thms]{Theorem}
\newtheorem{lemma}[thms]{Lemma}
\newtheorem{corollary}[thms]{Corollary}
\newtheorem{proposition}[thms]{Proposition}
\theoremstyle{definition}
\newtheorem{definition}[thms]{Definition}
\theoremstyle{remark}
\newtheorem{remark}[thms]{Remark}
\newtheorem*{observation*}{Observation}
\newtheorem*{example*}{Example}

\newtheoremstyle{strategy}{}{}{\itshape}{}{\bfseries}{}{.5em}{\thmname{#1} \thmnote{#3}. }
\theoremstyle{strategy}
\newtheorem*{strategy}{Strategy}

\newcommand{\partref}{\ref}
\numberwithin{equation}{section}
\numberwithin{figure}{section}

\newcommand{\strgy}{\mathfrak{S}}  

\begin{document}

\title[Dual-Pivot Quicksort: Optimality, Analysis and Zeros of Lattice Paths]{%
  Dual-Pivot Quicksort: Optimality, Analysis and 
  Zeros of Associated Lattice Paths}

\author[M.~Aumüller]{Martin Aumüller}
\address{Martin Aumüller, IT University of Copenhagen, Rued Langgaards Vej 7, 2300 Copenhagen, Denmark}
\email{\href{mailto:maau@itu.dk}{maau@itu.dk}}

\author[M.~Dietzfelbinger]{Martin Dietzfelbinger}
\address{Martin Dietzfelbinger, Fakultät für Informatik und Automatisierung, Technische Universität Ilmenau, Helmholtzplatz 5, 98693 Ilmenau, Germany}
\email{\href{mailto:martin.dietzfelbinger@tu-ilmenau.de}{martin.dietzfelbinger@tu-ilmenau.de}}

\author[C.~Heuberger]{Clemens Heuberger}
\address{Clemens Heuberger,
  Institut f\"ur Mathematik, Alpen-Adria-Universit\"at Klagenfurt,
  Universit\"atsstra\ss e 65--67, 9020 Klagenfurt am W\"orthersee, Austria}
\email{\href{mailto:clemens.heuberger@aau.at}{clemens.heuberger@aau.at}}

\author[D.~Krenn]{Daniel Krenn}
\address{Daniel Krenn,
  Institut f\"ur Mathematik, Alpen-Adria-Universit\"at Klagenfurt,
  Universit\"atsstra\ss e 65--67, 9020 Klagenfurt am W\"orthersee, Austria}
\email{\href{mailto:math@danielkrenn.at}{math@danielkrenn.at} \textit{or}
  \href{mailto:daniel.krenn@aau.at}{daniel.krenn@aau.at}}

\author[H.~Prodinger]{Helmut Prodinger}
\address{Helmut Prodinger, Department of Mathematical Sciences, Stellenbosch University, 7602 Stellenbosch,
 South Africa}
\email{\href{mailto:hproding@sun.ac.za}{hproding@sun.ac.za}}

\thanks{C.~Heuberger and D.~Krenn are supported by the
   Austrian Science Fund (FWF): P\,24644-N26
   and by the Karl Popper Kolleg ``Modeling--Simulation--Optimization''
   funded by the Alpen-Adria-Universität Klagenfurt and
   by the Carinthian Economic Promotion Fund (KWF).\\
   H.~Prodinger is supported by an incentive grant of the National
   Research Foundation of South Africa.}

 \thanks{We thank the reviewers for carefully reading the manuscript and
   for their suggestions.}

 \thanks{An extended abstract containing the ideas of the asymptotic
   analysis of the dual-pivot quicksort strategies ``Count'' and
   ``Clairvoyant'', as well as the lattice path analysis of this
   article appeared
   as~\cite{Aumueller-Diezfelbinger-Heuberger-Krenn-Prodinger:2016:quicksort-paths-arxiv},
   and an appendix containing proofs is available
   as~\href{http://arxiv.org/abs/1602.04031v1}{arXiv:1602.04031v1}. \\
   This article contains additionally a proof that ``Count'' is
   indeed the optimal strategy. This led to a major restructuring;
   the analysis now focuses on this strategy. Moreover some proofs
   were simplified.}


\begin{abstract}
  We present an average case analysis of a variant of dual-pivot
  quicksort. We show that the used algorithmic partitioning strategy 
  is optimal, i.e., it minimizes the expected number of key comparisons.
  For the analysis,
  we calculate the expected number of comparisons exactly as well as
  asymptotically, in particular, we provide exact expressions for the
  linear, logarithmic, and constant terms.
  
	An essential step is the analysis of zeros of lattice paths in a
  certain probability model. Along the way a combinatorial identity is
  proven.
\end{abstract}

\maketitle


\section{Introduction}
\label{sec:intro}


Dual-pivot quicksort~\cite{Sedgewick:1975:thesis,WildNN15,AumullerD15} 
is a family of sorting algorithms related to the
well-known quicksort algorithm. 
In order to sort an input sequence $(a_1,\ldots,a_n)$
of distinct elements, dual-pivot quicksort algorithms work as follows. 
(For simplicity we forbid repeated elements in the input.)
If $n \leq 1$, there is nothing to do.
If $n\ge2$, two of the input elements are selected as pivots.
Let $p$ be the smaller and $q$ be the larger pivot. 
The next step is to partition the remaining elements into
\begin{itemize}
\item the elements smaller than $p$ (``small elements''),
\item the elements between $p$ and $q$ (``medium elements''), and
\item the elements larger than $q$ (``large elements'').
\end{itemize}
Then the procedure is applied recursively to these three groups to complete the sorting. 

The cost measure used in this work is the number of comparisons between elements.
As is common, we will assume the input sequence is in random order, 
which means that each permutation of the $n$ elements occurs with
probability $\nicefrac{1}{n!}$\,.
With this assumption we may, without loss of generality, choose $a_1$
and $a_n$ as the pivots. 
Even in this setting there are different dual-pivot quicksort algorithms;
their difference lies in the way the partitioning is organized,
which influences the partitioning cost. 
More specifically, when a non-pivot element is considered, the strategy has to decide
whether it is compared with $p$ first or with $q$ first.
This is in contrast to standard quicksort with one pivot, where the
partitioning strategy does not influence the cost---in partitioning always
one comparison is needed per non-pivot element.
In dual-pivot quicksort, the average cost (over all permutations) of partitioning
and of sorting can be analyzed only when the partitioning strategy is fixed.

Only in 2009, Yaroslavskiy, Bentley, and Bloch~\cite{Yaroslavskiy-Mailinglist} described a dual-pivot
quicksort algorithm that makes $1.9 n \log n + O(n)$ key comparisons~\cite{WildNN15}.%
\footnote{In this paper ``$\log$'' denotes the natural logarithm to base $e$.}
This beats the classical quicksort algorithm~\cite{Hoare62}, which needs
$2(n+1)H_n-4n = 2n \log n + O(n)$ comparisons on average. (Here $H_n$ denotes the $n$-th harmonic number. We refer to the book
\cite[5.2.2 (24)]{Knuth:1998:Art:3} for this asymptotic as well as
an exact result.) Wild's Ph.D.~thesis~\cite{Wild:2016:Phd-dual-pivot} discusses implications of the Yaroslavskiy--Bentley--Bloch-algorithm (YBB) and many variants in detail. 

In \cite{AumullerD15}, the first two authors of this article described
the full design space for dual-pivot quicksort algorithms with respect to
counting element comparisons.
Among others, they studied a special partitioning 
strategy called ``Count'': when classifying a new element this strategy uses the large pivot $q$ for the first comparison 
if and only if among the elements seen up to this point the
number of large elements exceeds the number of small ones.
On the basis of a suitable probability model it can be argued that the relative frequencies
of small and large elements seen so far are maximum-likelihood estimators 
for the probability of the next element being small respectively large. 
In this sense this strategy seems quite natural. 
It was shown {in~\cite{AumullerD15}} that dual-pivot quicksort carries out $1.8 n \log n + O(n)$
key comparisons on average when this partitioning strategy is used and that
the main term in this formula is optimal/minimal. (They showed that no other
partitioning strategy can have a smaller main term, even if this strategy
has access to a certain oracle, cf.\ their strategy ``Clairvoyant''.)

One purpose of this paper is to make the expected number of
comparisons in the algorithm ``Count'' precise,
both for partitioning and for the resulting dual-pivot quicksort
variants. Moreover, we will show that ``Count'' is optimal among
all algorithmic strategies; see Part~\partref{sec:count-is-optimal} for details.

Already in \cite{AumullerD15} it was noted that 
the exact value of the expected partitioning cost
(i.e., the number of comparisons) of the mentioned strategy
depends on the expected number of the zeros of certain lattice paths
(Part~\partref{sec:lattice-paths}).
A complete understanding of this situation is the basis for our analysis
of dual-pivot quicksort, which appears in Part~\partref{sec:quicksort}.


\section{Overview and Results}
\label{sec:results}


This work is split into three parts. We give a brief overview on the
main results of each of these parts here.

In order to formulate the main results, we have to fix some notation.
We use Iverson's convention
\begin{equation*}
  \iverson{\mathit{expr}} =
  \begin{cases}
    1&\text{ if $\mathit{expr}$ is true},\\
    0&\text{ if $\mathit{expr}$ is false},
  \end{cases}
\end{equation*}
popularized by Graham, Knuth, and Patashnik~\cite{Graham-Knuth-Patashnik:1994}.

The harmonic numbers and variants will be denoted by
\begin{equation*}
  H_n = \sum_{m=1}^n \frac{1}{m}, \qquad
  \Hodd_n = \sum_{m=1}^n \frac{\iverson*{$m$ odd}}{m}
  \qquad\text{and}\qquad
  \Halt_n = \sum_{m=1}^n \frac{(-1)^m}{m}.
\end{equation*}
Of course, there are relations between these three definitions such as
$\Halt_n=H_n - 2\Hodd_n$ and $\Hodd_n + H_{\floor{n/2}}/2=H_n$, but it will
turn out to be much more convenient to use all three notations.
Asymptotically, $H_n = \log n + \Oh{1}$, $\Hodd_n = (\log n)/2+\Oh{1}$, and $\Halt_n=-\log2+\Oh{1/n}$; see Lemma \ref{lem:harmonic-asy} for details.

We denote  multinomial coefficients by
\begin{equation*}
  \binom{n}{c_1,c_2,\ldots,c_h} = \frac{n!}{c_1!\,c_2!\dots c_h!}.
\end{equation*}


\subsection*{Part~\partref{sec:lattice-paths}:
  \nameref{sec:lattice-paths}}

In the first part we analyze the behavior of certain random lattice paths of a fixed
length~$n$.
They model a particular aspect of the partition procedure
of dual-pivot quicksort.
The probability model is as follows:
A path starts at the origin~$(0,0)$,
goes steps $(1,+1)$ and $(1,-1)$, and
stops after $n$ steps;
the ending point
is chosen from the set $\set{(n,-n),(n,-n+2),\dots,(n,n-2),(n,n)}$
of feasible points uniformly at random.
For each ending point,
all paths from the origin to this point
are equally likely. We are interested in
the number of zeros, denoted by the random
variable~$Z_n$, of such paths.

Lattice path enumeration has a long tradition. An early reference is
\cite{Mohanty:1979:lattic};
a recent survey paper on the subject is \cite{Krattenthaler:2015:lattic}.
We emphasize that the probability model of the lattice paths studied in this
paper differs from the standard one where all paths of equal length are equally likely.

An exact formula for the expected number $\E{Z_n}$ of zeros is derived
in two different ways (see identity~\eqref{eq:id-intro} for these formul\ae):
On the one hand, we use the symbolic method and
generating functions (see Section~\ref{sec:gf}), which gives the
result in form of a double sum (Theorem~\ref{thm:paths-gf-zeros}).
This machinery extends well to higher moments and also allows us to
obtain the distribution. The exact distribution is given in
Theorem~\ref{thm:distribution}; its limiting behavior
is given by a discrete distribution: we have
\begin{equation*}
  \lim_{n\to\infty} \P{Z_n = r} = \frac{1}{r(r+1)}.
\end{equation*}

On the other hand, a more probabilistic approach gives the 
expectation~$\E{Z_n}$ as the simple single sum
\begin{equation}\label{eq:simple-sum-intro}
  \E{Z_n} = \sum_{m=1}^{n+1} \frac{\iverson*{$m$ odd}}{m}=\Hodd_{n+1};
\end{equation}
see Section~\ref{sec:prob} for more details. From this, the
asymptotic behavior $\E{Z_n} = \frac12 \log n + \Oh{1}$ can be extracted
(Section~\ref{sec:asymptotics}).

The two approaches give rise to the identity
\begin{equation}\label{eq:id-intro}
  \sum_{m=1}^{n+1} \frac{\iverson*{$m$ odd}}{m}
  =
  \frac{4}{n+1}
  \sum_{0\leq k < \ell < \ceil{n/2}} \frac{\binom{n}{k}}{\binom{n}{\ell}}
  + \iverson*{$n$ even} \frac{1}{n+1}
  \biggl(\frac{2^n}{\binom{n}{n/2}} - 1\biggr) + 1;
\end{equation}
the double sum~\eqref{eq:id-intro} equals the
single sum~\eqref{eq:simple-sum-intro} of
Theorem~\ref{thm:paths-prob} by combinatorial considerations. One
might ask about a direct proof of this interesting combinatorial identity. This can be achieved
by methods related to hypergeometric sums; the computational proof
is presented in Section~\ref{sec:identity}. We also provide a completely
elementary proof which is ``purely human''.


\subsection*{Part~\partref{sec:quicksort}:
  \nameref{sec:quicksort}}

One main result of this work analyzes key comparisons in the dual-pivot
quicksort algorithm that uses the optimal (see
Part~\partref{sec:count-is-optimal}) partitioning strategy
``Count''. Aumüller and Dietzfelbinger showed in \cite{AumullerD15}
that this algorithm requires
$1.8 n \log n + O(n)$ comparisons on average, which improves on 
the average number of comparisons in quicksort ($2n \log n + (2\gamma-4)n
+ 2\log n + O(1)$) and the 
recent dual-pivot algorithm of Yaroslavskiy et al.\@ ($1.9 n \log n + O(n)$; see \cite{WildNN15}). However, for real-world input sizes $n$ the (usually negative) factor in the linear term has a great influence on the comparison count. Our asymptotic result is stated as the following theorem.

\begin{theorem*}
  The average number of comparisons in the dual-pivot quicksort
  algorithm with a comparison-optimal partitioning strategy is
  \begin{equation*}
    \frac{9}{5} n \log n + A n + B \log n + C + \Oh{1/n}
  \end{equation*}
  as $n$ tends to infinity, with
  \begin{equation*}
    A = \frac95\gamma
    + \frac{1}{5} \log 2
    - \frac{89}{25}
    = -2.382\dots
  \end{equation*}
\end{theorem*}
with the Euler--Mascheroni constant~$\gamma = 0.5772156649\!\dots\,$.
The constants $B$ and $C$ are explicitly given, too, and more
terms of the asymptotics are presented. The precise result is formulated
as Corollary~\ref{cor:count:cost:asy}.

In fact, we even get an exact expression for the average comparison
cost. The precise result is formulated as
Theorem~\ref{thm:count:cost}. The same analysis can be
carried out for the non-algorithmic---it has access to an oracle---partitioning strategy
``Clairvoyant''~\cite{AumullerD15}; see Appendix~\ref{sec:clairvoyant} for the result.


\subsection*{Part~\partref{sec:count-is-optimal}:
  \nameref{sec:count-is-optimal}}

In Aumüller and Dietzfelbinger~\cite{AumullerD15} it was shown that
the strategy ``Clairvoyant'', which has access to an oracle to predict
the number of small and large elements in the remaining list,
minimizes the number of key comparisons among all such strategies with
oracle; thus it is called optimal. The main result of
Part~\partref{sec:count-is-optimal} is that the algorithmic partitioning
strategy ``Count'' is an optimal strategy among all algorithmic
dual-pivot partitioning strategies. This means that the analysis from
Part~\partref{sec:quicksort} yields an exact and sharp lower bound for
the average number of key comparisons
of arbitrary dual-pivot quicksort algorithms (Theorem~\ref{thm:count-optimal}).


\section{Background: Random Lattice Paths in $\N_0^h$}
\label{sec:lattice-paths-delta}

In this paper, we use two types of lattice paths with very similar
properties. We collect the relevant definitions and results for both situations
in this section.

Let $h\ge 1$ be an integer and $n\ge 0$. In this section, we
consider lattice paths in $\N_0^h$. All lattice paths of this section start in the origin and
the admissible steps are  $e_1=(1, 0,\ldots, 0)^\top$, \ldots, $e_h=(0,
\ldots, 0, 1)^\top$.  For any lattice path $v$ with $n$ steps and $0\le t\le n$, we write
$v_{\le t}$ for the lattice path consisting of the first $t$ steps of $v$. By
construction $v_{\le t}$ ends in $\Omega_t\coloneqq \{ c\in \N_0^h \mid c_1+\dots+c_h = t\}$.

We consider a random lattice path $V$ with $n$ steps
under the following probability model:
\begin{itemize}
\item All endpoints in $\Omega_n$ are equally likely.
\item For $c\in\Omega_n$, all lattice paths ending in $c$ are equally likely.
\end{itemize}

\begin{lemma}\label{lemma:k-dimensional-lattice-paths}
  For $0 \le t\le n$,  $c\in\Omega_t$, and $1\le j\le h$ we have
  \begin{align*}
    \P{V_{\le t}\text{ ends in $c$}}&=\frac{1}{\abs{\Omega_t}}=\frac1{\binom{t+h-1}{h-1}},\\
    \P{V_{\le t+1}\text{ ends in $c+e_j$}\mid V_{\le t}\text{ ends in
    $c$}}&=\frac{c_j+1}{t+h} \quad\text{ for $t<n$},
  \end{align*}
  where, as above, $V_{\le t}$ denotes the initial segment of length $t$ of $V$.
\end{lemma}
\begin{proof}
  Consider another random lattice path $V'$,
	which is defined as a Markov chain with transition probabilities
  \begin{equation}\label{eq:W-prime-definition}
     \P{V'_{\le t+1}\text{ ends in $c+e_j$}\mid V'_{\le t}\text{ ends in
    $c$}}=\frac{c_j+1}{t+h},
  \end{equation}
  for $0\le t < n$, $c\in \Omega_t$, and $1\le j \le h$.
	The Markov condition means that in fact
	$\P{V'_{\le t+1}\text{ ends in $c+e_j$}\mid V'_{\le t}=v'}=\frac{c_j+1}{t+h}$
	for each single lattice path $v'$ of length $t$ that ends in $c$.
  We investigate $V'$ with the aim of proving that $V$ and $V'_{\le n}$
  are identically distributed.

  Let $0\le t\le n$ and $v'$ be some lattice path with $n$ steps. Assume that
  $v'_{\le t}$ ends in $c\in\Omega_t$. We claim that
  \begin{equation}\label{eq:lattice-path-path-probability}
    \P{V'_{\le t}=v'_{\le t}}=\frac1{\binom{t+h-1}{h-1}}\frac{1}{\binom{t}{c_1,c_2,\ldots,c_h}}.
  \end{equation}
  We prove this by induction on $t$. The claim is certainly true for $t=0$ because
  $\Omega_0=\{0\}$. Now let $0<t<n$ and let $e_j$ be the last step of $v'_{t+1}$. Then we have
  \begin{align*}
    \P{V'_{\le t+1}=v_{\le t+1}'}&=\P{V'_{\le  t+1}\text{ ends in $c+e_j$}\mid 
                V'_{\le t} = v'_{\le t}}\P{V_{\le t}'=v_{\le t}'}\\
                          &=\frac{c_j+1}{t+h}\frac1{\binom{t+h-1}{h-1}}\frac{1}{\binom{t}{c_1,c_2,\ldots,c_h}}\\
                          &=    \frac{1}{\frac{t+h}{t+1}\binom{t+h-1}{t}\frac{t+1}{c_j+1}\binom{t}{c_1,c_2,\ldots,c_h}}\\
    &=\frac{1}{\binom{t+h}{t+1}}\frac{1}{\binom{t+1}{c_1, \ldots, c_j+1,
        \ldots, c_h}},
  \end{align*}
  which is \eqref{eq:lattice-path-path-probability} for $t+1$.

  Note that \eqref{eq:lattice-path-path-probability} implies that
  $\P{V'_{\le t}=v'_{\le t}}$ only depends on the endpoint $c$ of $v'_{\le t}$ and not on the
  initial segment $v'_{\le t}$ itself. Thus all  $v'_{\le t}$ ending in
  $c$  are equally likely. There are $\binom{t}{c_1,c_2,\ldots,c_h}$
  lattice paths of length $t$ ending in $c$,
  thus~\eqref{eq:lattice-path-path-probability} implies that
  \begin{equation}\label{eq:w-prime-end}
    \P{V_{t}'\text{ ends in $c$}} = \frac{1}{\binom{t+h-1}{h-1}}=\frac{1}{\abs{\Omega_t}}.
  \end{equation}
  This probability does not depend on $c\in\Omega_t$, thus all endpoints $c$
  after $t$ steps are equally likely.

  For $t=n$, the last two observations imply that $V'_{\le n}$ and $V$ are identically
  distributed. Thus the assertions of the lemma follow from
  \eqref{eq:w-prime-end} and \eqref{eq:W-prime-definition}.
\end{proof}

Note that the random lattice path model considered in this proof is another formulation of
a contagion P{\'o}lya urn model~\cite{Mahmoud:2008:Polya-urn-models}
(also called P{\'o}lya--Eggenberger urn model)
with balls of $h$ colors. Initially, there is one ball of
each color in the urn. In each step, a ball is drawn at random and is replaced by two
balls of the same color. The color of the ball determines the next step of $V'$.

\part{Lattice Paths}
\label{sec:lattice-paths}


As explained in the introduction, our analysis of the optimal partitioning procedure of dual pivot quicksort is based on a
certain lattice path model. The quantity of interest is the number of zeros of these lattice paths. This first
part is devoted to a thorough analysis of these lattice paths and its zeros.

The lattice paths have positive and negative
ordinate values and a fixed length $n$; they are introduced in Section~\ref{sec:description} by a
precise description of the probabilistic model. 
We emphasize that this probability model is different from the most commonly used
model, where all lattice paths are equally likely. We will work with this
model throughout Part~\partref{sec:lattice-paths}. In
Section~\ref{sec:description}, we give a precise description of our probability
model and define the
random variable~$Z_n$ counting the number
of zeros of the lattice paths.

In the following sections, we determine the value of $\E{Z_n}$ both exactly
(Sections~\ref{sec:gf} and~\ref{sec:prob}), as well
as asymptotically (Section~\ref{sec:asymptotics}). In Section~\ref{sec:gf}, we
use the machinery of generating functions. This machinery turns out to be
overkill if we are just interested in the expectation $\E{Z_n}$. However, it
easily allows extension to higher moments and the limiting distribution.

In Section~\ref{sec:prob}, we follow a
probabilistic approach, which first gives a result on
the probability model:
the equidistribution at the final values turns out to
carry over to every fixed length initial segment of the path. This result yields
a very simple expression for the expectation
$\E{Z_n}$ in terms of harmonic numbers, and
thus immediately yields a precise asymptotic expansion for $\E{Z_n}$.
This distributional result is a consequence of the results on more general lattice paths
considered in Section~\ref{sec:lattice-paths-delta}.
The generating function approach,
however, gives the expectation in terms of a double sum of quotients of
binomial coefficients (the right-hand side of \eqref{eq:id-intro}).

Section~\ref{sec:identity} gives a direct computational
proof that these two results coincide. The original expression in
\cite{AumullerD15} (a double sum over a quotient
of a product of binomial coefficients and a binomial coefficient) is also shown
to be equal to our identity; see Section~\ref{sec:identity}.
Both explicit and asymptotic expressions for the distribution
$\P{Z_n=r}$ can be found in Section~\ref{sec:distribution}.


\section{Probabilistic Model}
\label{sec:description}


We consider paths of a given length~$n$ on the lattice $\Z^2$, where only \emph{up-steps}
$(1,+1)$ and \emph{down-steps} $(1,-1)$ are allowed. These paths are chosen at random according to the rules
below.

Let us fix a length~$n\in\N_0$. A path\footnote{We denote the lattice paths by
  the symbol~$W_n$ (whose shape is close to a visualization
  of a lattice path); the more natural symbol~$P_n$ is used later for
  the partitioning cost. In contrast to the lattice paths in
  Section~\ref{sec:lattice-paths-delta}, we include the length~$n$
  in the notation~$W_n$
  because we consider various random variables depending on~$W_n$
  for $n\to\infty$.}~$W_n$ starting in the origin~$(0,0)$ (no
choice for this starting point) is chosen according to the following rules.
\begin{itemize}
\item First, we choose an ending point $(n,D)$, where $D$ is a random
  integer uniformly distributed in $\{-n, -n+2, \ldots, n-2, n\}$, i.\,e., $D=d$ occurs only for
  integers $d$ with $\abs{d}\le n$ and $d\equiv n \pmod 2$.
\item Second, a path is chosen uniformly at random among all paths from $(0,0)$
  to $(n,D)$.
\end{itemize}
The conditions on $D$ characterize those ending points that are reachable from~$(0,0)$.
It is easy to see that the lattice paths~$W_n$ and $V$ of
Section~\ref{sec:lattice-paths-delta} are closely related; see the
proof of Lemma~\ref{lem:prob-point-uniform} for details.

We are interested in the number of intersections of
a path with the horizontal axis.
To make this precise, we define a \emph{zero} of a path $W_n$ as a point
$(x,0)\in W_n$. 

Thus, let $W_n$ be a path of length~$n$ which is chosen according to the
probabilistic model above and define the random variable
\begin{equation*}
  Z_n = \text{number of zeros of $W_n$}.
\end{equation*}


\section{Using the Generating Function Machinery}
\label{sec:gf}


\begin{theorem}\label{thm:paths-gf-zeros}
  The expected number of zeros in a randomly
  (as described in Section~\ref{sec:description}) chosen path of
  length~$n$ is
  \begin{equation*}
    \E{Z_n} =
    \frac{4}{n+1}
    \sum_{0\leq k < \ell < \ceil{n/2}} \frac{\binom{n}{k}}{\binom{n}{\ell}}
    + \iverson*{$n$ even} \frac{1}{n+1}
    \left(\frac{2^n}{\binom{n}{n/2}} - 1\right) + 1.
  \end{equation*}
\end{theorem}

The remaining part of this section is devoted to the proof of this theorem.
The main technique is to model the lattice paths by means of combinatorial
classes and generating functions, i.e., the symbolic method; see, for example, Flajolet and
Sedgewick~\cite{Flajolet-Sedgewick:ta:analy}.

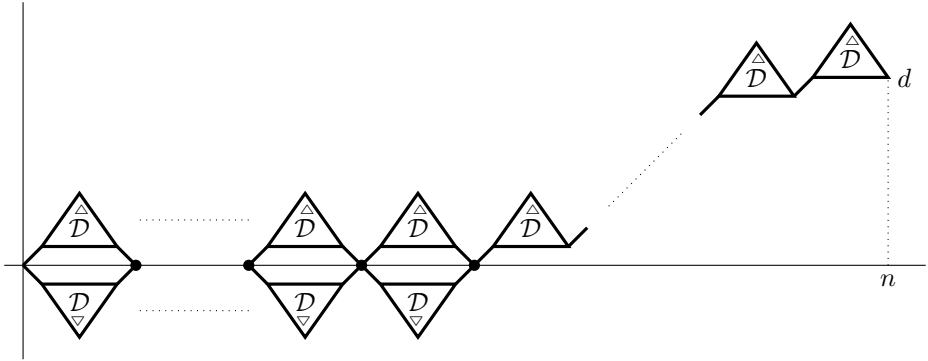
\begin{figure}
  \centering
  \bgroup
  \renewcommand{\CZero}{}
  \input{figure-decomp-leavings}
  \egroup
  \caption{Decomposition of a lattice path for $d\geq0$ marking zeros by the symbol~$\bullet$\,.}
  \label{fig:decomp-path-zeros}
\end{figure}

In Figure~\ref{fig:decomp-path-zeros}, we give a schematic decomposition of a
path from $(0, 0)$ to $(n, d)$ for non-negative $d$. We denote the classes of a
single ascent $\Upstep$ by $\CUpstep$ and a single descent $\Downstep$ by $\CDownstep$. The class of Dyck-paths (paths starting and ending at the same
height, but not going below it) is denoted by $\CDyckpath$.  A reflection of a
Dyck-path is denoted by $\CReflectedDyckpath$. A zero (except the
first at the origin) is represented by the singleton class $\CZero$. Note that we do not mark
the zero at $(0,0)$ for technical reasons; we will take this into account at the
end by adding a $1$ to the final result.

With these notations, this decomposition
can be described as follows (the path is read
from the left to the right).
\begin{itemize}
\item We start at $(0,0)$ by
  either doing a single ascent $\Upstep$, followed by a Dyck path of $\CDyckpath$
  and a single descent $\Downstep$ or doing a single descent $\Downstep$,
  followed by a reflected Dyck path of $\CReflectedDyckpath$
  and a single ascent $\Upstep$. Thus, we have reached a zero $(x,0)$.
\item We mark this zero by $\CZero$.
\item We repeat such up or down blocks $\CUpblock \coloneqq \CUpstep\Ctimes\CDyckpath\Ctimes\CDownstep$ or $\CDownblock \coloneqq \CDownstep\Ctimes\CReflectedDyckpath\Ctimes\CUpstep$, each one followed by a
  zero~$\CZero$, a finite number of times.
\item We end by $d$ consecutive blocks of $\CDyckpath$, each preceded by a single
  ascent $\CUpstep$. This gives the paths to their end point at height~$d$.
  Thus, there is a block $\CTerminalUpblock \coloneqq \CUpstep\Ctimes\CDyckpath$
  for the last step $\Upstep$ on each level.
\end{itemize}
Written as a symbolic expression, this decomposition amounts to
\begin{equation*}
  \f[Big]{\textsc{Seq}_{\geq0}}{
    \bigl(\CUpblock \,\cup\, \CDownblock\bigr)\Ctimes\CZero}
  \Ctimes
  \f[Big]{\textsc{Seq}_{=d}}{\CTerminalUpblock}.
\end{equation*}

Now turning to generating functions, we mark a step to the right (i.e., up- and down-steps) by
the variable~$z$ and a zero (except at the origin) by $u$.
Thus, the coefficient of $z^nu^{r-1}$ of
the function $Q_d(z,u)$ (the generating function of all paths starting in
$(0,0)$ and ending in $(n,d)$ for some~$n$) equals the number of paths of length~$n$
that have exactly $r$ zeros.

Thus the generating functions corresponding to the classes $\CUpstep$,
$\CDownstep$ and $\CZero$ are $z$, $z$ and $u$, respectively.
The generating
function~$\f{D}{z}$ corresponding to the class of Dyck-paths $\CDyckpath$ equals
\begin{equation*}
  \f{D}{z} = \frac{1-\sqrt{1-4z^2}}{2z^2}
\end{equation*}
(see~\cite[Example~I.16 and page 6]{Flajolet-Sedgewick:ta:analy}); we have to
replace $z$ by $z^2$ because the number of Dyck paths of length $n$ equals the
number of binary trees with $n/2$ inner nodes. It is clear that $\f{D}{z}$ also
corresponds to reflected Dyck-paths $\CReflectedDyckpath$.

The decomposition above
translates to the generating function
\begin{equation}
  \label{eq:path-gf-zeros}
  Q_d(z,u) = \frac{\f{D}{z}^{\abs{d}} z^{\abs{d}}}{1 - 2 u z^2 \f{D}{z}},
\end{equation}
which we will use from now on. If $d<0$, then the construction is the same, but everything is reflected at the
horizontal axis and \eqref{eq:path-gf-zeros} remains valid (as we already wrote
$\abs{d}$).

To obtain a nice explicit form, we perform a change of variables. The result is
stated in the following lemma.

\begin{lemma}\label{lem:transform-gf-zeros}
  With the transformation $z = v / (1+v^2)$, which is valid for
  $z$ (and $v$) in a suitable neighborhood of zero, we have
  \begin{equation*}
    Q_d(z,u) = \frac{v^{\abs{d}}(1+v^2)}{1-v^2(2u-1)}.
  \end{equation*}
\end{lemma}

\begin{proof}
  Transforming the counting generating function of Dyck paths yields
  \begin{equation*}
    \f{D}{z} = 1+v^2.
  \end{equation*}
  Thus~\eqref{eq:path-gf-zeros} becomes
  \begin{equation*}
    Q_d(z,u) = 
    (1+v^2)^{\abs{d}}
    \Big(\frac{v}{1+v^2}\Big)^{\abs{d}}
    \frac{1}{1-2u\big(\frac{v}{1+v^2}\big)^2(1+v^2)},
  \end{equation*}
  which can be simplified to the expression stated in the lemma.
\end{proof}

The next step is to extract the coefficients out of the expressions obtained in
the previous lemma. First we rewrite the extraction of the coefficients from
the ``$z$-world'' to the ``$v$-world''; see
Lemma~\ref{lem:extract-coeffs-worlds}. Afterwards, in
Lemma~\ref{lem:coeffs-zeros}, the coefficients can be determined quite easily.

\begin{lemma}\label{lem:extract-coeffs-worlds}
  Let $F(z)$ be an analytic function in a neighborhood of the origin. Then we have
  \begin{equation*}
    [z^n] F(z) = [v^n] (1-v^2) (1+v^2)^{n-1}
      \f{F}{\frac{v}{1+v^2}}.
  \end{equation*}
\end{lemma}

\begin{proof}
  We use Cauchy's formula to extract the coefficients of $F(z)$ as
  \begin{equation*}
    [z^n] F(z) = \frac{1}{2\pi i}\oint_\calC F(z) \frac{\dd z}{z^{n+1}}
  \end{equation*}
  where $\calC$ is a positively oriented small circle around the origin. Under
  the transformation $z=v/(1+v^2)$, the circle $\calC$ is transformed to a
  contour $\calC'$ which still winds exactly once around the origin. Using
  Cauchy's formula again, we obtain
  \begin{align*}
    [z^n] F(z)
    &= \frac{1}{2\pi i}\oint_{\calC'} \f{F}{\frac{v}{1+v^2}}
    \frac{(1+v^2)^{n+1}}{v^{n+1}} \frac{1-v^2}{(1+v^2)^2} \,\dd v \\
    &= [v^n] (1-v^2) (1+v^2)^{n-1} \f{F}{\frac{v}{1+v^2}},
  \end{align*}
  which was to be shown.
\end{proof}

Now we are ready to calculate the desired coefficients.

\begin{lemma}\label{lem:coeffs-zeros}
  Suppose $n\equiv d\pmod 2$. Then we have
  \begin{equation*}
    [z^n] Q_d(z,1)
    = \binom{n}{(n-d)/2}
  \end{equation*}
  and, moreover,
  \begin{equation*}
    [z^n] \left.\frac{\partial}{\partial u} Q_d(z,u)\right\vert_{u=1}
    = 2 \sum_{k=0}^{(n-\abs{d})/2-1}\binom{n}{k}.
  \end{equation*}
\end{lemma}

\begin{proof}
  As $n \equiv d \pmod 2$, the number $n-d$ is even, and so we can
  set $\ell = \frac12(n-d)$. Then $[z^n]Q_d(z, 1)$ is the number of paths from
  $(0, 0)$ to $(n, d)$. These paths have $n-\ell$ up-steps and $\ell$ down-steps;
  thus there are $\binom{n}{\ell}$ many such paths.

  For the second part of the lemma, we restrict ourselves to $d\geq0$ (otherwise use $-d$ and the symmetry in~$d$
  of the generating function~\eqref{eq:path-gf-zeros} instead). We start with the result of
  Lemma~\ref{lem:transform-gf-zeros}. Taking the first derivative and setting
  $u=1$ yields
  \begin{equation*}
    \left.\frac{\partial}{\partial u} Q_d(z,u)\right\vert_{u=1}
    = \frac{2v^{d+2}(1+v^2)}{(1-v^2)^2}.
  \end{equation*}
  Thus, by using Lemma~\ref{lem:extract-coeffs-worlds}, we get
  \begin{equation*}
    [z^n] \frac{2v^{d+2}(1+v^2)}{(1-v^2)^2}
    = 2\,[v^{n-d-2}]\frac{(1+v^2)^{n}}{1-v^2}.
  \end{equation*}
  We use $\ell$ as above and get
  \begin{equation*}
    [v^{n-d-2}]\frac{(1+v^2)^n}{1-v^2}
    = [v^{2\ell-2}]\frac{(1+v^2)^n}{1-v^2}
    = [v^{\ell-1}] \frac{(1+v)^n}{1-v}
    = \sum_{k=0}^{\ell-1}\binom{n}{k}
  \end{equation*}
  as desired.
\end{proof}

We are now ready to prove the main theorem (Theorem~\ref{thm:paths-gf-zeros})
of this section, which provides an expression for the expected number of
zeros in our random lattice paths.
This exact expression is written as a double sum.
\begin{proof}[Proof of Theorem~\ref{thm:paths-gf-zeros}]
  By Lemma~\ref{lem:coeffs-zeros}, the average number of zeros
  (except the zero at the origin) of
  a path of length~$n$ which ends in $(n,d)$ is
  \begin{equation*}
    \mu_{n,d} = \frac{[z^n] \left.\frac{\partial}{\partial u}
        Q_d(z,u)\right\vert_{u=1}}{[z^n] Q_d(z,1)}
    = \frac{2}{\binom{n}{\ell}} \sum_{k=0}^{\ell-1}\binom{n}{k},
  \end{equation*}
  where we have set $\ell = \frac12 (n-\abs{d})$ as in the proof of
  Lemma~\ref{lem:coeffs-zeros}. If $d=0$, this simplifies to
  \begin{equation}\label{eq:mu_n_0}
    \mu_{n,0} = \frac{2}{\binom{n}{n/2}} \sum_{k=0}^{n/2-1}\binom{n}{k}
    = \frac{2^n}{\binom{n}{n/2}} - 1.
  \end{equation}
  If $n\not\equiv d\pmod 2$, then we set $\mu_{n,d} = 0$.

  Summing up yields
  \begin{align*}
    \sum_{d=-n}^n \mu_{n,d}
    &= 2 \sum_{d=1}^n \mu_{n,d} + \mu_{n,0}
    = 4 \sum_{\ell=0}^{\ceil{n/2}-1}
    \frac{1}{\binom{n}{\ell}} \sum_{k=0}^{\ell-1}\binom{n}{k}
    + \mu_{n,0} \\
    &= 4 \sum_{0\leq k < \ell < \ceil{n/2}} \frac{\binom{n}{k}}{\binom{n}{\ell}}
    + \iverson*{$n$ even} 
    \left(\frac{2^n}{\binom{n}{n/2}} - 1\right).
  \end{align*}
  Dividing by the number $n+1$ of possible end points and adding
  $1$ for the zero at the origin completes the proof of
  Theorem~\ref{thm:paths-gf-zeros}.
\end{proof}


\section{A Probabilistic Approach}
\label{sec:prob}


\begin{theorem}\label{thm:paths-prob}
  The expected number of zeros in a randomly
  (as described in Section~\ref{sec:description}) chosen path of
  length~$n$ is
  \begin{equation*}
    \E{Z_n} = \Hodd_{n+1}.
  \end{equation*}
\end{theorem}

In the analysis of
the quicksort algorithm in Part~\partref{sec:quicksort}, we need an
\emph{up-from-zero situation}, which is a point $(t, 0) \in W_n$ such
that $(t+1, 1) \in W_n$. Define the random variable
 \begin{equation*}
  Z^{\nearrow}_n = \text{number of up-from-zero situations of $W_n$}.
\end{equation*}
The following corollary provides the expected value of this random variable.

\begin{corollary}\label{cor:exp-up-from-zero}
  The expected number of up-from-zero situations in a randomly
  (as described in Section~\ref{sec:description}) chosen path of
  length~$n$ is
  \begin{equation*}
    \E{Z^{\nearrow}_n} = \tfrac12 \Hodd_n.
  \end{equation*}
\end{corollary}

In order to prove the theorem and the corollary, we need the following
property of our paths.

\begin{lemma}\label{lem:prob-point-uniform}
  Let $t\in\N_0$ with $t\leq n$. The probability that a random path $W_n$ (as
  defined in Section~\ref{sec:description}) runs through $(t,k)$ is
  \begin{equation}\label{eq:probability}
    \P{(t, k)\in W_n} = \frac{1}{t+1}
  \end{equation}
  for all $k$ with $\abs{k}\leq t$ and
  $k\equiv t \pmod 2$, otherwise it is $0$.
\end{lemma}

Our lattice path model of this part is equivalent to the
lattice paths model of Section~\ref{sec:lattice-paths-delta} with
dimension $h=2$.

\begin{proof}[Proof of Lemma~\ref{lem:prob-point-uniform}]
  We use the notation of Section~\ref{sec:lattice-paths-delta}. Let
  $h=2$. We identify a lattice path~$W_n$ with a lattice path~$V$ of
  Section~\ref{sec:lattice-paths-delta} in the following way: We
  identify up-steps with $e_1$ and down-steps with $e_2$. (This
  corresponds to a rotation of a lattice path~$W_n$ by $45$ degrees
  counterclockwise and a scaling by $1/\sqrt{2}$; see
  Figure~\ref{fig:lattice-paths-Wn-V}.)

  A point $(t, k)\in W_n$ then corresponds to a point in $\Omega_t$,
  and the uniform distribution follows by Lemma~\ref{lemma:k-dimensional-lattice-paths}.
\end{proof}

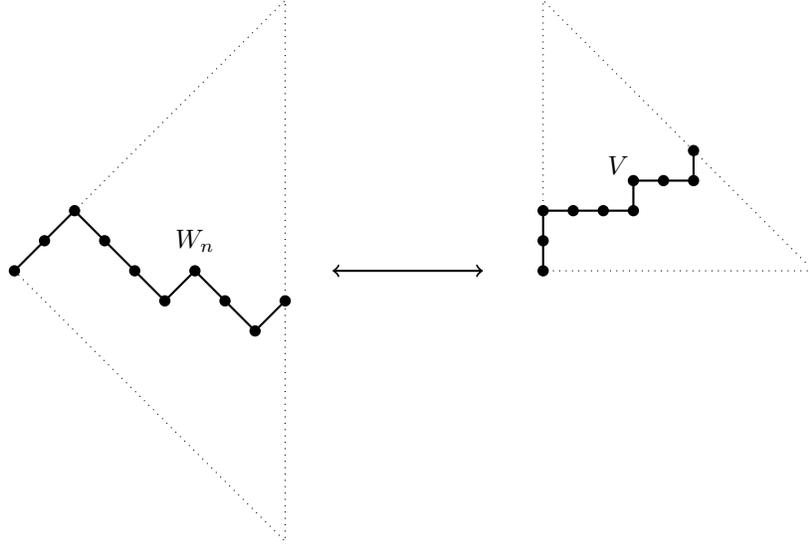
\begin{figure}
  \centering
  \begin{tikzpicture}[scale=0.4]
    \def\Wnpath{
      \draw [dotted] (0,0) -- (9,9) -- (9,-9) -- cycle;
      \draw [thick] (0,0) -- (1,1) -- (2,2) -- (3,1)
      -- (4,0) -- (5,-1)
      -- (6,0) -- (7,-1) -- (8,-2) -- (9,-1)
      ;
      \foreach \pt in {(0,0), (1,1), (2,2), (3,1),
        (4,0), (5,-1), (6,0), (7,-1), (8,-2), (9,-1)}
      \node [circle, fill, inner sep=0, minimum size=0.15cm] at \pt {};
    }

    \begin{scope}[xshift=-500]
      \Wnpath
      \node at (6,1) {$W_n$};
    \end{scope}
    \begin{scope}[rotate=45, scale=1/sqrt(2)]
      \Wnpath
      \node at (6,1) {$V$};
    \end{scope}

    \draw [<->, thick] (-7,0) -- (-2,0);
  \end{tikzpicture}
  \caption{Correspondence between lattice paths $W_n$ and $V$.}
  \label{fig:lattice-paths-Wn-V}
\end{figure}

\begin{proof}[Proof of Theorem~\ref{thm:paths-prob}]
  By Lemma~\ref{lem:prob-point-uniform}, the expected number of zeros of
  the path~$W_n$ is
  \begin{equation*}
    \E{Z_n} =
    \sum_{t=0}^{n} \P{(t,0) \in W_n}
    = \sum_{t=0}^{n} \iverson*{$t$ even} \frac{1}{t+1}
    = \sum_{t=1}^{n+1} \frac{\iverson*{$t$ odd}}{t}
    = \Hodd_{n+1},
  \end{equation*}
  which completes the proof of the theorem.
\end{proof}

\begin{proof}[Proof of Corollary~\ref{cor:exp-up-from-zero}]
  An up-from-zero situation can only occur at a zero of $W_n$. We have
  to exclude $(n,0)$. By Lemma~\ref{lemma:k-dimensional-lattice-paths} going up or down
  after a zero is equally likely because a zero of $W_n$ corresponds to a
  diagonal element in a lattice path in the setting of
  Section~\ref{sec:lattice-paths-delta}. Thus
  \begin{equation*}
    \E{Z^\nearrow_n} =
    \frac12\left(\E{Z_n} - \frac{\iverson*{$n$ even}}{n+1}\right)
    = \frac12 \Hodd_n,
  \end{equation*}
  which we wanted to show.
\end{proof}


\section{Identity}
\label{sec:identity}


Using the previous two sections we can show the following identity in a
combinatorial way.

\begin{theorem}\label{thm:identity}
  For $n\ge 0$, the following four expressions are equal:
  \begin{subequations}
  \begin{align}
    &\mathrel{\phantom{=}}
    \frac{4}{n+1}
    \sum_{0\leq k < \ell < \ceil{n/2}} \frac{\binom{n}{k}}{\binom{n}{\ell}}
    + \iverson*{$n$ even} \frac{1}{n+1}
    \left(\frac{2^n}{\binom{n}{n/2}} - 1\right) + 1
    \label{eq:id:double}\\
    &= \frac{2}{\floor{n/2}+1}
      \sum_{0\leq k < \ell \leq \floor{n/2}}
      \frac{\binom{2\floor{n/2}+1}{k}}{\binom{2\floor{n/2}+1}{\ell}} + 1
      \label{eq:id:double-simple}\\
    &= \frac{1}{n+1} \sum_{m=0}^{\floor{n/2}} \sum_{\ell=m}^{n-m}
    \frac{\binom{2m}{m} \binom{n-2m}{\ell-m}}{\binom{n}{\ell}}
    \label{eq:id:quicksort}\\
    &= \Hodd_{n+1}.
    \label{eq:id:single}
  \end{align}
  \end{subequations}
\end{theorem}

We note that the expressions \eqref{eq:id:double} and
  \eqref{eq:id:double-simple} are obviously equal for odd $n$. Furthermore,
  \eqref{eq:id:double-simple} and
  \eqref{eq:id:single} only change when $n$ increases by $2$
  (to be precise from odd~$n$ to even~$n$).
Once we prove that \eqref{eq:id:double}
equals \eqref{eq:id:single} for all $n\ge 0$, it follows that both expressions
are equal to \eqref{eq:id:double-simple} for all $n\ge 0$.

\begin{proof}[Combinatorial proof of Theorem~\ref{thm:identity}]
  First, we combine the results of Theorems~\ref{thm:paths-gf-zeros}
  and~\ref{thm:paths-prob} to see that \eqref{eq:id:double} and \eqref{eq:id:single}
  are equal.

  Expression~\eqref{eq:id:quicksort} for the expected number of zeros
  has been given in \cite[displayed
  equation after (14)]{AumullerD15}: The number of
  paths from $(0,0)$ to $(n, n-2\ell)$ is $\binom{n}{\ell}$, whereas the number of
  paths from $(0,0)$ via $(2m,0)$ to $(n, n-2\ell)$ is
  $\binom{2m}{m} \binom{n-2m}{\ell-m}$. Summing over all possible $m$ and all
  possible $\ell$ and dividing by $n+1$ for the equidistribution of the
  end point yields \eqref{eq:id:quicksort}.
\end{proof}

Beside this combinatorial proof, we intend to show Theorem~\ref{thm:identity}
in alternative ways. First, we prove that \eqref{eq:id:quicksort} equals \eqref{eq:id:single}:
  Consider
  \begin{align*}
    \frac{1}{n+1}\sum_{\ell=m}^{n-m}\frac{\binom{2m}{m}\binom{n-2m}{\ell-m}}{\binom{n}{\ell}}&=
    \frac{1}{n+1}\sum_{\ell=m}^{n-m}\frac{(2m)!\, (n-2m)!\, \ell!\, (n-\ell)!}{m!\, m!\, (\ell-m)!\, (n-\ell-m)!\, n!}\\
    &=\frac{1}{(n+1)\binom{n}{2m}}\sum_{\ell=m}^{n-m}\binom{\ell}{m}\binom{n-\ell}{m}\\
    &=\frac{1}{(n+1)\binom{n}{2m}} \binom{n+1}{2m+1}=\frac{1}{2m+1},
  \end{align*}
  where \cite[(5.26)]{Graham-Knuth-Patashnik:1994}\footnote{\label{footnote:GKP:5.26}
    Identity \cite[(5.26)]{Graham-Knuth-Patashnik:1994} regards
    the summation of products of binomial coefficients
    \begin{equation*}
      \bgroup
      \newcommand{\fbinom}[2]{\Bigl(\genfrac{}{}{0pt}{}{#1}{#2}\Bigr)}
      \sum_{0 \leq k \leq \ell} \fbinom{\ell-k}{m} \fbinom{q+k}{n}
      = \fbinom{\ell+q+1}{m+n+1}
      \egroup
    \end{equation*}
    for integers $\ell \geq 0$, $m \geq 0$, $n \geq q \geq 0$.
    This is used above with $k \mapsto \ell$, $\ell \mapsto n$, $n \mapsto m$,
    $m \mapsto m$, $q\mapsto0$.}
  has been used in the
  penultimate step. Summing over
  $m$ yields
  \begin{equation*}
    \sum_{m=0}^{\floor{n/2}} \frac{1}{2m+1} = \Hodd_{n+1},
  \end{equation*}
  thus the equality between \eqref{eq:id:quicksort} and \eqref{eq:id:single}.

It remains to give a computational proof of the equality
between~\eqref{eq:id:double} and \eqref{eq:id:single}. We provide two proofs:
one motivated by ``creative telescoping'' (Section~\ref{sec:creative-telescoping}) and one completely
elementary ``human'' proof (Section~\ref{sec:human-proof}) using not more than Vandermonde's convolution.

\subsection{Proof of the Identity Using Creative Telescoping}
\label{sec:creative-telescoping}
 A computational proof of the identity between \eqref{eq:id:double} and \eqref{eq:id:single} can be generated
by the
summation package Sigma~\cite{Schneider:2015:Sigma-1.81} (see also
Schneider~\cite{Schneider:2007:symb-sum}) together with the packages
HarmonicSums~\cite{Ablinger:2015:HarmonicSums-1.0} and
EvaluateMultiSums~\cite{Schneider:2015:EvaluateMultiSums-0.96}.\footnote{The authors thank Carsten Schneider for providing the packages
  Sigma~\cite{Schneider:2015:Sigma-1.81} and
  EvaluateMultiSums~\cite{Schneider:2015:EvaluateMultiSums-0.96}, and
  for his support.}
To succeed, we have to split the case of even and odd $n$. The obtained proof
certificates are rather lengthy to verify.

Motivated by these observations, we also give a proof without computer
support. Anyhow, the key step is, as with Sigma, creative telescoping. We
prove an (easier) identity, suggested by Sigma, in the following lemma.

\begin{lemma}
  Let
  \begin{align*}
    F(n, \ell)&=\sum_{0\le k<\ell}\frac{\binom{n}{k}}{\binom{n}{\ell}},\\
    G(n, \ell)&=(\ell-1)+ (\ell-1-n)F(n, \ell)
  \end{align*}
  for $0\le \ell\le n$.
  Then
  \begin{equation}\label{eq:CT:simple}
    (n+1)F(n+1, \ell)-(n+2)F(n, \ell)=G(n, \ell+1)- G(n, \ell)
  \end{equation}
  holds for all $0\le \ell< n$.
\end{lemma}

\begin{proof}
  For $0\le \ell<n$, we first compute
  \begin{align}
    F(n+1, \ell)&=\frac{1}{\binom{n+1}{\ell}}\sum_{0\le k<\ell}\binom{n+1}{k}
    =\frac{1}{\binom{n+1}{\ell}}\sum_{0\le k<\ell}\biggl(\binom{n}{k-1}+\binom{n}{k}\biggr)\notag\\
    &=-\frac{\binom{n}{\ell-1}}{\binom{n+1}{\ell}}+\frac{2\binom{n}{\ell}}{\binom{n+1}{\ell}}F(n,
    \ell)
    =-\frac{\ell}{n+1}+2\frac{n+1-\ell}{n+1}F(n,
    \ell)\label{eq:CT:simple:10}
  \end{align}
  and
  \begin{align}
    F(n, \ell+1)&=\frac1{\binom{n}{\ell+1}}\sum_{0\le k<\ell+1}\binom{n}{k}
    =\frac{\binom{n}{\ell}}{\binom{n}{\ell+1}}+
    \frac{\binom{n}{\ell}}{\binom{n}{\ell+1}}F(n, \ell)\notag\\
    &=\frac{\ell+1}{n-\ell} + \frac{\ell+1}{n-\ell}F(n, \ell).\label{eq:CT:simple:01}
  \end{align}
  By plugging \eqref{eq:CT:simple:10} and \eqref{eq:CT:simple:01} into
  \eqref{eq:CT:simple}, all occurrences of $F(n, \ell)$ cancel as well as all
  other terms, which proves \eqref{eq:CT:simple}.
\end{proof}

We are now able to prove the essential recurrence for the sum \eqref{eq:id:double} in
Theorem~\ref{thm:identity}.

\begin{lemma}\label{lemma:CT:recursion-E}
  Let
  \begin{equation*}
    E_n = \frac{4}{n+1}
    \sum_{0\leq k < \ell < \ceil{n/2}} \frac{\binom{n}{k}}{\binom{n}{\ell}}
    + \iverson*{$n$ even} \frac{1}{n+1}
    \left(\frac{2^n}{\binom{n}{n/2}} - 1\right).
  \end{equation*}
  Then
  \begin{equation}\label{eq:CT:recursion-E}
    E_{2N+1}-E_{2N}=0
    \qquad\text{and}\qquad
    E_{2N+2}-E_{2N+1}=\frac1{2N+3}
  \end{equation}
  for $N\ge 0$.
\end{lemma}

\begin{proof}
Multiplying \eqref{eq:CT:simple} with $4/((n+1)(n+2))$ and summing over $0\le
\ell < \ceil{n/2}$ yields
\begin{multline*}
  \frac{4}{n+2}\sum_{0\le k<\ell<
    \ceil{n/2}}\frac{\binom{n+1}{k}}{\binom{n+1}{\ell}}
  -\frac{4}{n+1}\sum_{0\le
    k<\ell<\ceil{n/2}}\frac{\binom{n}{k}}{\binom{n}{\ell}}\\
  =\frac{4}{(n+1)(n+2)}\Bigl(G(n, \ceil{n/2})-G(n, 0)\Bigr)
\end{multline*}
for $n\ge 0$. This is equivalent to
\begin{multline}\label{eq:CT:identity:3}
  \frac{4}{n+2}\sum_{0\le k<\ell<
    \ceil{(n+1)/2}}\frac{\binom{n+1}{k}}{\binom{n+1}{\ell}} - \frac{4[n\text{ even}]}{n+2}F(n+1,n/2)
  -\frac{4}{n+1}\sum_{0\le
    k<\ell<\ceil{n/2}}\frac{\binom{n}{k}}{\binom{n}{\ell}}\\
  =\frac{4}{(n+1)(n+2)}\Bigl(\ceil{n/2}-1+(\ceil{n/2}-1-n)F(n, \ceil{n/2})+1\Bigr).
\end{multline}
We rewrite the double sums in terms of $E_n$ and $E_{n+1}$, respectively, and
use \eqref{eq:mu_n_0}. We also replace
$F(n+1, n/2)$ by \eqref{eq:CT:simple:10}. Then \eqref{eq:CT:identity:3}
is equivalent to
\begin{multline}\label{eq:identity-expectations}
  E_{n+1}-\frac{2[n\text{ odd}]}{n+2}F(n+1, (n+1)/2)-E_n\\  +\frac{2n[n\text{ even}]}{(n+1)(n+2)}-\frac{2[n\text{ even}]}{n+1}F(n, n/2)\\
 =\frac{4}{(n+1)(n+2)}\Bigl(\ceil{n/2}-(\floor{n/2}+1)F(n, \ceil{n/2})\Bigr).
\end{multline}
If $n=2N+1$, equation~\eqref{eq:identity-expectations} is equivalent to
\begin{multline*}
  E_{2N+2}-E_{2N+1}-\frac{2}{2N+3}F(2N+2,N+1)\\
  =\frac{2}{2N+3}-\frac{2}{2N+3}F(2N+1, N+1).
\end{multline*}
Using \eqref{eq:CT:simple:10}, this is equivalent to the second recurrence in \eqref{eq:CT:recursion-E}.

If $n=2N$, then \eqref{eq:identity-expectations} is equivalent to the first
recurrence in \eqref{eq:CT:recursion-E}.
\end{proof}

\begin{proof}[Computational proof of Theorem~\ref{thm:identity}]
  The definition of $E_0$ in Lemma~\ref{lemma:CT:recursion-E} implies that
  $E_0=0$. Thus \eqref{eq:id:double} and \eqref{eq:id:single}
  coincide for $n=0$. This can be extended to all $n\ge 0$ by induction on $n$
  and Lemma~\ref{lemma:CT:recursion-E}.
\end{proof}

\subsection{Proof of the Identity Using Vandermonde's Convolution}
\label{sec:human-proof}

\begingroup\newcommand{\m}{\mathchoice{\ceil{\frac n2}}{\ceil{\frac n2}}{\lceil\!\frac n2\!\rceil}{\lceil\!\frac n2\!\rceil}}
We now provide a completely elementary ``human'' proof of the identity between
\eqref{eq:id:double} and \eqref{eq:id:single}.

We first prove an identity on binomial coefficients.

\begin{lemma}\label{lemma:binomial-identity} The identity
\begin{equation*}
\sum_{0\le k\le j}\binom {n}{k}
=\sum_{0\le k \le j}2^k\binom{n-k-1}{j-k}
\end{equation*}
holds for all non-negative integers $j<n$.
\end{lemma}
\begin{proof}
We denote the right hand side by $\rho$. The binomial theorem and symmetry of the
binomial coefficient yield
\begin{align*}
\rho &=\sum_{0\le i\le  k \le j}\binom{k}{i}\binom{n-k-1}{n-1-j}.
\intertext{The sum over $k$ can be evaluated by \cite[(5.26)]{Graham-Knuth-Patashnik:1994} (see also the footnote on page~\pageref{footnote:GKP:5.26} for this identity) resulting in}
\rho &=\sum_{0\le i \le j}\binom{n}{n+i-j}.
\intertext{Symmetry of the binomial coefficient and then replacing $i$ by $j-i$
lead to}
\rho &=\sum_{0\le i \le j}\binom{n}{j-i}=\sum_{0\le i \le j}\binom{n}{i}.
\end{align*}
\end{proof}

We are now able to 
establish a recurrence satisfied 
by the double sum in \eqref{eq:id:double}.

\begin{lemma}\label{lemma:recursion-binomial}For $n\ge 0$, let
  \begin{equation*}
    S_n=\sum_{0\le k< \ell< \m} \frac{\binom{n}{k}}{\binom{n}{\ell}}.
  \end{equation*}
  Then the recurrence
  \begin{equation*}
    S_n =  \frac{n+1}{n-1}S_{n-2}+\frac{n+1}{4n}
    + \iverson*{$n$ even}\Bigl(\frac{2^{n-2}}{n\binom{n}{n/2}} -
    \frac{1}{2(n-1)} - \frac{1}{4n} \Bigr)
  \end{equation*}
  holds for $n\ge 2$.
\end{lemma}
\begin{proof}
We replace the sum over $k$ in $S_n$ by the term found in Lemma~\ref{lemma:binomial-identity}
and obtain
\begin{align}
S_n&=\sum_{0\le k\le \ell< \m} \frac{\binom{n}{k}}{\binom{n}{\ell}}-\m\notag\\
&=\sum_{0\le k\le \ell< \m} \frac{2^k\binom{n-k-1}{\ell-k}}{\binom{n}{\ell}}-\m\notag\\
&=\sum_{0\le k\le \ell< \m} \frac{2^k(n-k-1)!\,k!}{n!}\binom{\ell}{k}\bigl((n-k)-(\ell-k)\bigr)-\m\notag\\
&=\sum_{0\le k\le \ell< \m} \frac{2^k(n-k)!\,k!}{n!}\binom{\ell}{k}
\notag\\&\qquad-\sum_{0\le k\le \ell< \m} \frac{2^k(n-k-1)!\,(k+1)!}{n!}\binom{\ell}{k+1}-\m.\notag
\intertext{In both sums, the sum over $\ell$ can be evaluated using upper summation (see \cite[Table~174]{Graham-Knuth-Patashnik:1994}), and we get}
S_n&=\sum_{0\le k< \m} \frac{2^k(n-k)!\,k!}{n!}\binom{\m}{k+1}
\notag\\&\qquad-\sum_{0\le k< \m} \frac{2^k(n-k-1)!\,(k+1)!}{n!}\binom{\m}{k+2}-\m.\notag
\intertext{Shifting the summation index in the second sum leads to}
S_n&=\sum_{0\le k< \m} \frac{2^k(n-k)!\, k!}{n!}\binom{\m}{k+1}
\notag\\&\qquad-\sum_{0\le k< \m} \frac{2^{k-1}(n-k)!\, k!}{n!}\binom{\m}{k+1}-\m+
\frac{1}{2}\m\notag\\
&=\frac{1}{2n!}\sum_{0\le k< \m} 2^k(n-k)!\, k!\binom{\m}{k+1}-\frac12\m\notag\\
&=\frac{\m!}{2n!}\sum_{0\le k< \m}
2^k\frac{(n-k)!}{(\m-k-1)!}\frac1{k+1}-\frac12\m\label{S_n_explicit}.
\end{align}
We intend to derive a recurrence linking $S_n$ with
$S_{n-2}$. Therefore, for $n\ge 2$, we rewrite $S_n$ as
\begin{align*}
  S_n&=\frac{\m!}{2n!}\sum_{0\le k< \m-1}
  2^k\frac{(n-k-2)!}{(\m-k-2)!}\frac{(n-k)(n-k-1)}{(\m-k-1)(k+1)} - \frac{\m}{2} +
  \frac{2^{\m-2}}{\binom{n}{\m-1}}.
\intertext{Partial fraction decomposition in $k$ yields}
  S_n&=\frac{\m!}{2n!}\sum_{0\le k< \m-1}
  2^k\frac{(n-k-2)!}{(\m-k-2)!}\Bigl( - 1+\frac{(n-\m)(n-\m+1)}{{(\m-k-1)} \m} \\&\hspace*{16.25em}+ \frac{n(n+1)}{{(k +
        1)} \m}\Bigr)
  - \frac{\m}{2} + \frac{2^{\m-2}}{\binom{n}{\m-1}}\\
  &=-\frac{\m!\, (n-\m)!}{2n!}\sum_{0\le k< \m-1}
  2^k\binom{n-k-2}{\m-k-2}\\
  &\qquad +\frac{(\m-1)!\, (n-\m+1)!}{2n!}\sum_{0\le k<
    \m-1}2^k\binom{n-k-2}{\m-k-1}\\
  &\qquad +\frac{(\m-1)!\, (n+1)n}{2n!}\sum_{0\le k< \m-1}
  2^k\frac{(n-k-2)!}{(\m-k-2)!}\frac{1}{{(k +
        1)}}\\
  &\qquad- \frac{\m}{2} + \frac{2^{\m-2}}{\binom{n}{\m-1}}.
\end{align*}
We again use Lemma~\ref{lemma:binomial-identity} for the first two summands and
\eqref{S_n_explicit} with $n$ replaced by $n-2$ for the third summand
to
obtain
\begin{align*}
  S_n&=-\frac{1}{2\binom{n}{\m}}\sum_{0\le k< \m-1}\binom{n-1}{k}
  +\frac{1}{2\binom{n}{\m-1}}\sum_{0\le k<
    \m}\binom{n-1}{k}\\
  &\qquad +\frac{n+1}{n-1}\Bigl(S_{n-2}+\frac{\m-1}2\Bigr)- \frac{\m}{2}.
\end{align*}
Adding another summand for $k=\m-1$ to the first sum leads to
\begin{align*}
  S_n&=\biggl(\frac{1}{2\binom{n}{\m-1}}-\frac{1}{2\binom{n}{\m}}\biggr)\sum_{0\le
    k< \m}\binom{n-1}{k}\\
   &\qquad
   +\frac{\binom{n-1}{\m-1}}{2\binom{n}{\m}}+\frac{n+1}{n-1}S_{n-2}+\frac{\m-1}{n-1}
   -\frac12\\
  &=\frac{n-2\m+1}{2\m\binom{n}{\m}}\sum_{0\le
    k< \m}\binom{n-1}{k}\\
   &\qquad +\frac{n+1}{n-1}S_{n-2}+\frac{\m}{2n}+\frac{\m-1}{n-1} -\frac12.
\end{align*}
If $n$ is odd, then $n-2\m+1=0$ so that
the first summand vanishes. The result follows in that case.

For even $n$, the remaining sum is $2^{n-2}$ and the result follows.
\end{proof}

\begin{proof}[Computational proof of Theorem~\ref{thm:identity}]
  Denote the expression in~\eqref{eq:id:double} by $E_n$. From
  Lemma~\ref{lemma:recursion-binomial}, we obtain the recurrence
  \begin{equation*}
    E_n = E_{n-2}+\frac{1}{n+\iverson*{$n$ even}}
  \end{equation*}
  for $n\ge 2$. As $E_0=E_1=1$, this implies that $E_n=\Hodd_{n+1}$. Thus
  \eqref{eq:id:double} and \eqref{eq:id:single} coincide.
\end{proof}

\endgroup

\section{Asymptotic Aspects}
\label{sec:asymptotics}
\begin{lemma}\label{lem:harmonic-asy}
  We have
  \begin{subequations}
  \begin{align}
    H_n &= \log n + \gamma + \frac{1}{2n} - \frac{1}{12n^2}
    + \Oh[Big]{\frac{1}{n^4}},\label{eq:asymptotic-harmonic-number}\\
    \Hodd_n &= \frac12 \log n
    + \frac{\gamma + \log 2}{2} + \frac{\iverson*{$n$ odd}}{2n}
    + \frac{3\iverson*{$n$ even}-2}{12n^2}
    + \Oh[Big]{\frac{1}{n^4}},\label{eq:asymptotic-odd-harmonic-number}\\
    \Halt_n &= -\log 2 + \frac{(-1)^n}{2n}
    - \frac{(-1)^n}{4n^2}
    + \Oh[Big]{\frac{1}{n^4}}\label{eq:asymptotic-alternating-harmonic-number}
  \end{align}
  \end{subequations}
  as $n$ tends to infinity.
\end{lemma}

Here, $\gamma = 0.5772156649\!\dots$ is the \emph{Euler--Mascheroni constant}.

\begin{proof}
  The asymptotic expansion~\eqref{eq:asymptotic-harmonic-number} is well-known,
  cf.\@ for instance \cite[(9.88)]{Graham-Knuth-Patashnik:1994}.

  We write $\alpha_n=\iverson*{$n$ even}$ and thus
  $\floor{n/2}=(n-1+\alpha_n)/2$. As $\alpha_n$ is obviously bounded, we can
  simply plug this expression into the asymptotic expansions and simplify all
  occurring higher powers of $\alpha_n$ by the fact that $\alpha_n^2=\alpha_n$. Using the relations
  $\Hodd_n=H_n- H_{\floor{n/2}}/2$ and $\Halt_n=H_n - 2\Hodd_n$ leads to the
  expansions~\eqref{eq:asymptotic-odd-harmonic-number} and
  \eqref{eq:asymptotic-alternating-harmonic-number}, respectively.

  The actual asymptotic computations\footnote{A worksheet containing
    the computations can be found at
    \url{http://www.danielkrenn.at/downloads/quicksort-paths-full/quicksort-paths.ipynb}.}
  have been carried out using the asymptotic
  expansions module \cite{Hackl-Krenn:2015:asy-sage} of
  SageMath~\cite{SageMath:2016:7.4}.
\end{proof}

\begin{corollary}\label{cor:main-asy}
  The expected numbers of zeros and up-from-zero situations are
  \begin{align*}
    \E{Z_n} &=
    \frac12 \log n + \frac{\gamma+\log2}{2}
    + \frac{1+\iverson*{$n$ even}}{2n}
    - \frac{2+9\iverson*{$n$ even}}{12n^2}
    + \frac{\iverson*{$n$ even}}{n^3}
    + \Oh[Big]{\frac{1}{n^4}},\\
    \E{Z^\nearrow_n} &=
    \frac14 \log n + \frac{\gamma+\log2}{4}
    + \frac{\iverson*{$n$ odd}}{4n}
    + \frac{3\iverson*{$n$ even} - 2}{24n^2}
    + \Oh[Big]{\frac{1}{n^4}},
  \end{align*}
  respectively, asymptotically as $n$ tends to infinity.
\end{corollary}

\begin{proof}
  Combine Theorem~\ref{thm:paths-prob} and 
  Corollary~\ref{cor:exp-up-from-zero} with Lemma~\ref{lem:harmonic-asy}.
\end{proof}


\section{Distribution}
\label{sec:distribution}

In this section, we study the distribution of the number
of zeros. As for the expectation $\E{Z_n}$, we get an exact
formula as well as an asymptotic formula.

\begin{theorem}\label{thm:distribution}
  Let $r\in\N$. For positive lengths~$n\geq 2r-2$, the probability that a
  randomly chosen path~$W_n$ has exactly~$r$ zeros is
  \begin{align*}
    \P{Z_n = r} &= \frac{2^r}{n+1}
    \frac{\binom{\ceil{n/2}}{r}}{\binom{n}{r}}
    \bigg( \frac{2\ceil{n/2}}{r(r+1)} + \frac{r-1}{r+1} +
        \iverson*{$n$ even} \frac{1}{r} \bigg) \\
    &\phantom{=}\;+ \iverson*{$n$ even}
    \frac{2^{r-1}(r-1)}{(n+1)r} \frac{\binom{n/2}{r-1}}{\binom{n}{r}}
  \end{align*}
  and we have $\P{Z_0 = r} = \iverson{r=1}$. Otherwise, if $1\le n<2r-2$, we have $\P{Z_n = r} = 0$.
\end{theorem}

This exact formula admits a local limit theorem towards a discrete
distribution.\footnote{This local limit theorem of the number of zeros
  might be used for a distributional analysis of the optimal
  dual-pivot partitioning strategy. This is beyond the scope of this
  article and we defer to future work here.}
The details are as follows.

\begin{corollary}\label{cor:distribution-asy}
  Let $0<\eps\leq\frac12$. For positive integers $r$ with $r =
  \Oh{n^{1/2-\eps}}$, we have asymptotically
  \begin{equation*}
    \P{Z_n = r} = \frac{1}{r(r+1)} \left(1 + \Oh{1/n^{2\eps}}\right)
  \end{equation*}
  as $n$ tends to infinity.
\end{corollary}

Note that the limiting distribution $\P{Z_\infty=r}=1/(r(r+1))$ is reminiscent of a
truncated version of the Zeta distribution with parameter $s=2$. While the
untruncated Zeta distribution has infinite mean, this is not the case for $Z_n$
due to the restriction $r=\Oh{n^{1/2-\eps}}$. However, $\E{Z_n}=\frac12 \log
n+\Oh{1}$ (see Corollary~\ref{cor:main-asy}) tends to infinity, as expected.

\begin{proof}[Proof of Theorem~\ref{thm:distribution}]
  Again, we assume $d\geq0$ (by symmetry of the generating
  function~\eqref{eq:path-gf-zeros}). Note that $Q_d$ counts the
  number of zeros by the variable~$u$ except for the first zero (at
  $(0,0)$). By starting with Lemma~\ref{lem:transform-gf-zeros} and
  some rewriting, we can extract the $(r-1)$st coefficient with
  respect to $u$ as
  \begin{equation*}
    [u^{r-1}] Q_d(z,u)
    = [u^{r-1}] \frac{v^d}{1-u\frac{2v^2}{1+v^2}}
    = \frac{2^{r-1} v^{2(r-1)+d}}{(1+v^2)^{r-1}}.
  \end{equation*}
  Next, we extract the coefficient of $z^n$. We use
  Lemma~\ref{lem:extract-coeffs-worlds} to obtain
  \begin{align*}
    [z^n u^{r-1}] Q_d(z,u)
    &= [v^n] (1-v^2) (1+v^2)^{n-1} \frac{2^{r-1} v^{2(r-1)+d}}{(1+v^2)^{r-1}} \\
    &=2^{r-1}\,[v^{n-d-2(r-1)}] (1-v^2) (1+v^2)^{n-r}.
  \end{align*}
  We set $\ell=\frac12(n-d)$ and get
  \begin{align*}
    [z^n u^{r-1}] Q_d(z,u)
    &= 2^{r-1}\,[v^{2\ell-2r+2}] (1-v^2) (1+v^2)^{n-r} \\
    &= 2^{r-1}\,[v^{\ell-r+1}] (1-v) (1+v)^{n-r} \\
    &= 2^{r-1} \binom{n-r}{\ell-r+1} - 2^{r-1} \binom{n-r}{\ell-r} \\
    &= 2^{r-1} \binom{n-r}{n-\ell-1} - 2^{r-1} \binom{n-r}{n-\ell}.
  \end{align*}
  Note that we have to assume $n-r\geq0$ to make this work. Otherwise,
  anyhow, there are no paths with exactly $r$ zeros (and positive
  length~$n$).

  If $\ell>r-1$ we can rewrite the previous formula to obtain
  \begin{equation*}
    [z^n u^{r-1}] Q_d(z,u)
    = 2^{r-1} \binom{n-r}{n-\ell} \left(\frac{n-\ell}{\ell-r+1} - 1\right)
    = 2^{r-1} \frac{d+r-1}{\ell-r+1} \binom{n-r}{n-\ell},
  \end{equation*}
  if $\ell=r-1$, then we have $[z^n u^{r-1}] Q_d(z,u) = 2^{r-1}$ (independently of~$n$),
  and if $\ell<r-1$ we get $[z^n u^{r-1}] Q_d(z,u) = 0$.

  To finish the proof, we have to normalize this number of paths with exactly $r$
  zeros and then sum up over all $\ell$. So let us start with the
  normalization part. We set
  \begin{equation*}
    \lambda_{n,r,d} = \P{Z_n = r \mid
      \text{$W_n$ ends in $(n,d)$}}
    = \frac{[z^n u^{r-1}] Q_d(z,u)}{[z^n] Q_d(z,1)}
  \end{equation*}
  for $n\equiv d\pmod 2$ and $\lambda_{n,r,d} = 0$ otherwise. The denominator
  $[z^n] Q_d(z,1)$ was already determined in Lemma~\ref{lem:coeffs-zeros}.

  If $\ell>r-1$, we have
  \begin{align*}
    \lambda_{n,r,d}
    &= 2^{r-1} \frac{d+r-1}{\ell-r+1} \binom{n-r}{n-\ell} \Big/ \binom{n}{\ell} \\
    &= 2^{r-1} \frac{d+r-1}{\ell-r+1} \frac{(n-r)!\,\ell!\,(n-\ell)!}{
      (n-\ell)!\,(\ell-r)!\,n!} \\
    &= 2^{r-1} \frac{(n-r)!}{n!} \frac{\ell!\,(d+r-1)}{(\ell-r+1)!},
  \end{align*}
  where the last line holds for $\ell=r-1$ as well. In particular, we
  obtain
  \begin{equation*}
    \lambda_{n,r,0}
    = 2^{r-1} (r-1) \iverson{n \geq 2r-2} \frac{(n/2)!}{n!} \frac{(n-r)!}{(n/2-r+1)!}
    = \frac{2^{r-1} (r-1)}{r} \frac{\binom{n/2}{r-1}}{\binom{n}{r}}.
  \end{equation*}

  We have arrived at the summation of the $\lambda_{n,r,d}$. The result
  follows as
  \begin{align*}
    \P{Z_n = r}
    &= \sum_{d=-n}^n \P{Z_n = r \mid \text{$W_n$ ends in $(n,d)$}}
    \P{\text{$W_n$ ends in $(n,d)$}} \\
    &= \frac{1}{n+1} \sum_{d=-n}^n \lambda_{n,r,d}\\
    &= \frac{2}{n+1} \sum_{\ell=0}^{\ceil{n/2}-1} \lambda_{n,r,n-2\ell}
    + \frac{1}{n+1} \lambda_{n,r,0},
  \end{align*}
  and plugging in $\lambda_{n,r,d}$ gives the intermediate result
  \begin{equation}\label{eq:distribution-intermediate}
  \begin{split}
    \P{Z_n = r}
    &= 2^{r} \frac{(n-r)!}{(n+1)!} \sum_{\ell=r-1}^{\ceil{n/2}-1}
     \frac{\ell!\,(n-2\ell+r-1)}{(\ell-r+1)!} \\
    &\phantom{=}\;+ \iverson*{$n$ even}
    \frac{2^{r-1}(r-1)}{(n+1)r} \binom{n/2}{r-1} \Big/ \binom{n}{r}
  \end{split}
  \end{equation}
  for $n\geq r$.

  To complete the proof, we have to show that
  \begin{equation}\label{eq:distribution-sum}
     \sum_{\ell=r-1}^{\ceil{n/2}-1}
     \frac{\ell!\,(n-2\ell+r-1)}{(\ell-r+1)!}
     = r!\, \binom{\ceil{n/2}}{r}
     \bigg( \frac{2\ceil{\frac{n}{2}}}{r(r+1)} + \frac{r-1}{r+1} +
       \iverson*{$n$ even} \frac{1}{r} \bigg)
  \end{equation}
  and to rewrite the factorials as binomial coefficients.

Of course, \eqref{eq:distribution-sum} can be proved computationally by, for example,
Sigma~\cite{Schneider:2015:Sigma-1.81}. However, we give a direct proof here.

  We obtain
  \begin{align*}
     \sum_{\ell=r-1}^{\ceil{n/2}-1}&
     \frac{\ell!\,(n-2\ell+r-1)}{(\ell-r+1)!}\\ &=
     (r-1)! \sum_{\ell=r-1}^{\ceil{n/2}-1}
     \binom{\ell}{r-1}(n-2\ell+r-1)\\
     &=(r-1)! \sum_{\ell=r-1}^{\ceil{n/2}-1} \binom{\ell}{r-1}(n+1+r-2(\ell+1))\\
     &=(n+1+r) (r-1)! \sum_{\ell=r-1}^{\ceil{n/2}-1} \binom{\ell}{r-1}-2r!
     \sum_{\ell=r-1}^{\ceil{n/2}-1} \binom{\ell+1}{r}.
     \intertext{Using upper summation,
       cf.~\cite[5.10]{Graham-Knuth-Patashnik:1994}, yields}
     \sum_{\ell=r-1}^{\ceil{n/2}-1}&
     \frac{\ell!\,(n-2\ell+r-1)}{(\ell-r+1)!}\\
     &=(n+1+r)(r-1)!\, \binom{\ceil{n/2}}{r} - 2r!\, \binom{\ceil{n/2}+1}{r+1}\\
     &=(r-1)!\, \binom{\ceil{n/2}}{r}\left(n+1+r - 2\frac{r}{r+1}(\ceil{n/2}+1)\right)\\
     &=(r-1)!\, \binom{\ceil{n/2}}{r}\left(n+r-1 -
       2\ceil{n/2}+\frac{2\ceil{n/2}}{r+1}+\frac{2}{r+1}\right).
  \end{align*}
  The identity \eqref{eq:distribution-sum} follows by replacing $n-2\ceil{n/2}$ with $\iverson*{$n$
      even} - 1$ and by collecting terms.
\end{proof}

As a next step, we want to prove Corollary~\ref{cor:distribution-asy}, which
extracts the asymptotic behavior of the distribution
(Theorem~\ref{thm:distribution}). To show that asymptotic formula, we
will use
the following auxiliary result.

\begin{lemma}\label{lem:quo-factorials}
 Let $0<\eps\leq\frac12$. For integers $c$ with $c =
 \Oh{N^{1/2-\eps}}$ we have
 \begin{equation*}
   c!\,\binom{N}{c}  = N^c(1+\Oh{1/N^{2\eps}}).
 \end{equation*}
\end{lemma}

\begin{proof}
The inequality $N^c \ge c!\,\binom{N}{c}$ is trivial. We observe
 \begin{align*}
   c!\,\binom{N}{c}
   &= N^{c}\cdot\prod_{0\le i < c}\left(1-\frac{i}{N}\right)
   \ge N^{c}\cdot\biggl(1-\sum_{0\le i < c}\frac{i}{N}\biggr)
   \ge N^c \left(1-\frac{c^2}{2N}\right)\\
   &= N^c \left(1+\Oh{1/N^{2\eps}}\right),
 \end{align*}
where the assumption on $c$ has been used in the last step.
\end{proof}

\begin{proof}[Proof of Corollary~\ref{cor:distribution-asy}]
  By using Lemma~\ref{lem:quo-factorials}, the exact result of
  Theorem~\ref{thm:distribution} becomes
  \begin{align*}
    \P{Z_n = r}
    &= \frac{2^r}{n} \frac{n^r}{2^r} \frac{1}{n^r}
    \left(\frac{n}{r(r+1)} + \Oh{1}\right)
    \left(1+\Oh{1/n^{2\eps}}\right) \\
    &\phantom{=}\;+ \iverson*{$n$ even}
    \frac{2^{r-1}(r-1)}{n} \frac{n^{r-1}}{2^{r-1}} \frac{1}{n^r}
    \left(1+\Oh{1/n^{2\eps}}\right) \\
    &= \frac{1}{r(r+1)} \left(1 + \Oh{1/n^{2\eps}}\right),
  \end{align*}
  as claimed.
\end{proof}



\part{Dual-Pivot Quicksort}
\label{sec:quicksort}

In this second part of this work, we analyze a
partitioning strategy and the dual-pivot quicksort
algorithm itself.

As mentioned in the introduction, the number of key comparisons of
dual-pivot quicksort depends on the concrete partitioning procedure.  This
is in contrast to the standard quicksort algorithm with only one
pivot, where partitioning always has cost $n-1$ (for partitioning $n$
elements; one is taken as the pivot).
For example, if one wants to classify a
large element, i.\,e., an element larger than the larger pivot, 
comparing it with the larger pivot is unavoidable,
but it depends on the partitioning procedure whether
a comparison with the smaller pivot occurs as well. 

First, in Section~\ref{sec:solve-recurrence},
we make the set-up precise, fix notions, and start solving the
dual-pivot quicksort recurrence~\eqref{eq:recurrence}. This recurrence
relates the cost of the partitioning step to the total number of
comparisons of dual-pivot quicksort. These results are independent of the
partition strategy.

In Section~\ref{sec:part-costs} the partitioning strategy ``Count'' is
introduced and its cost is analyzed. 
It will turn out that the results on lattice paths obtained in Part~\partref{sec:lattice-paths} are central in determining the partitioning cost.

Everything is put together in Section~\ref{sec:costs-main-asy}:
We obtain the exact comparison count 
(Theorem~\ref{thm:count:cost}). 
The asymptotic behavior is extracted out of the exact results
(Corollary~\ref{cor:count:cost:asy}).

\section{Solution of the Dual-Pivot Quicksort Recurrence}
\label{sec:solve-recurrence}

We consider versions of dual-pivot quicksort that act as follows
on an input sequence $(a_1,\ldots,a_n)$ consisting of distinct numbers:
If $n\le1$, do nothing, otherwise choose $a_1$ and $a_n$ as pivots,
and by one comparison determine $p=\min(a_1,a_n)$ and $q=\max(a_1,a_n)$.
Use a partitioning procedure to partition the remaining
$n-2$ elements into the three classes \emph{small}, \emph{medium}, and
\emph{large}. Then call dual-pivot quicksort recursively on each of these three classes
to finish the sorting, using the same partitioning procedure in all recursive calls.  

We denote the number of small, medium and large elements defined by
the pivot elements by $S$, $M$ and $L$ respectively.

Of course, the random variables $S$, $M$ and $L$ depend on $n$. For simplicity and readability,
this is not reflected in the notation. However, our discussion heavily uses
generating functions, where we need to
make the dependence on $n$ explicit; so we will write $S_n$, $M_n$ and $L_n$ in
this context.

We will need the following results on the distribution of $(S, M, L)$.

\begin{lemma}\label{lemma:S-M-L-distribution} The triple $(S, M, L)$ is uniformly distributed on 
  \begin{equation}\label{eq:def-Omega}
    \Omega'\coloneqq\{(s, m, \ell)\mid \text{$s+m+\ell=n-2$, $s, m, \ell\in\N_0$}\}.
  \end{equation}
  The random variables $S$, $M$ and $L$ are identically distributed and we have
  \begin{align}
    \P{M=m}&=\frac{n-m-1}{\binom{n}{2}}\quad\text{for $0\le m\le n-2$},\label{eq:distribution-M}\\
    \E{M}&=\frac{n-2}{3}, \label{eq:expectation-M}\\
    \sum_{n\ge 2}\E{\abs{L_n-S_n}}z^n &=
    \frac{1}{3  {(1-z)}^{2}}
    - \frac{1}{2  {(1- z)}}
    - \frac{1}{2} {(z + 1)} \artanh(z)
    + \frac{1}{6}
    + \frac{1}{3}  z
\label{eq:expectation-difference}.
  \end{align}
\end{lemma}
\begin{proof}
  Each pair of pivot elements is equally likely and there is an obvious bijection
  between the pairs of pivot elements and the triples $(S, M,
  L)\in\Omega'$. Thus $(S, M, L)$ is uniformly distributed on $\Omega'$. This
  (or direct enumeration of $\Omega'$) implies that
  $\Omega'$ has cardinality $\binom{n}{2}$.

  By the symmetry in the definition of $\Omega'$, it is clear that $S$, $M$ and
  $L$ are identically distributed. Thus $3\E{M}=\E{S}+\E{M}+\E{L}=\E{n-2}=n-2$,
  which implies~\eqref{eq:expectation-M}.
  For $0\le m\le n-2$, there are $n-m-1$ choices for $(s, \ell)$ such that $(s,
  m,\ell)\in\Omega'$. Together with $\abs{\Omega'}=\binom{n}{2}$, this implies
  \eqref{eq:distribution-M}.

  Let $g_{d,n}$ be the number of triples $(s, m, \ell)\in\Omega'$ with
  $\abs{\ell-s}=d$. We can write the bivariate generating function of $g_{d,n}$
  as
  \begin{equation*}
    G(u, z)\coloneqq\sum_{\substack{n\ge 0\\d\ge 0}}g_{d,n}u^dz^n=\frac{1}{1-z^2}\frac{1}{1-z}\Bigl(1+\frac{2uz}{1-uz}\Bigr).
  \end{equation*}
  To see this, first only consider the cases that $\ell>s$. Write
  $s+m+\ell=s+m+(s+(\ell-s))=2s+m+(\ell-s)$. As $\ell-s$ is marked by the
  variable $u$, this triple $(s, m, \ell)$ contributes $(z^2)^s z^m
  (zu)^{\ell-s}$. Summing over all these triples yields $1/(1-z^2)\cdot
  1/(1-z)\cdot uz/(1-uz)$. The triples with $\ell<s$ contribute the same, thus the
  factor $2$ in the result. The triples with $\ell=s$ contribute $1/(1-z^2)\cdot 1/(1-z)$.

  Therefore,
  \begin{equation}\label{eq:expectation-L-S-extract}
    \E{\abs{L_n-S_n}}=\frac1{\binom{n}{2}}[z^{n-2}]\frac{\partial G(u, z)}{\partial
                   u}\Bigr|_{u=1}=
                   \frac2{n(n-1)} [z^{n-2}]\frac{2z}{(1+z)(1-z^4)}.
  \end{equation}
  We have to compute the generating function
  \begin{equation*}
    H(z)\coloneqq \sum_{n\ge 2} \E{\abs{L_n-S_n}}z^n.
  \end{equation*}
  Differentiating twice yields and using~\eqref{eq:expectation-L-S-extract} yields
  \begin{equation*}
    H''(z) = \sum_{n\ge 2}n(n-1)\E{\abs{L_n-S_n}}z^{n-2} = \frac{4z}{(1+z)(1-z^4)}.
  \end{equation*}
  Integrating twice yields~\eqref{eq:expectation-difference}.
\end{proof}

Let $P_n$, a random 
variable, denote the \emph{partitioning cost}. This is 
defined as the number of comparisons made by the partitioning procedure
if the input $(a_1,\ldots,a_n)$ is assumed to be in random order.
Further, let $C_n$ be the random variable
that denotes the number of comparisons carried out when sorting $n$ elements
with dual-pivot quicksort.
The reader should be aware that both $P_n$ and $C_n$ are determined by the partitioning
procedure used.  

Since the input $(a_1,\ldots,a_n)$ is in random order
and the partitioning procedure does nothing but compare elements with the 
two pivots, the inputs for the recursive calls are in random order as well,
which implies that the distributions of $P_n$ and $C_n$
only depend on $n$.

The recurrence
\begin{equation}
  \E{C_n} =
     \E{P_n}  + \frac{3}{\binom{n}{2}}
      \sum_{k = 1}^{n - 2} (n - 1 - k) \E{C_{k}}
\label{eq:recurrence}
\end{equation}
for $n\ge0$
describes the connection between
the expected sorting cost $\E{C_n}$ and the expected partitioning cost $\E{P_n}$.
It will be central for our analysis. 
Note that it is irrelevant for \eqref{eq:recurrence} how the partitioning cost $\E{P_n}$ is determined;
it need not even be referring to comparisons. 
The recurrence is simple and well-known; a version of it occurs already in 
Sedgewick's thesis~\cite{Sedgewick:1975:thesis}.
For the convenience of the reader we give a brief proof: We clearly
have
  $C_n = P_n + C_S + C_M + C_L$.
Taking expectations and using the fact that $S$, $M$, $L$ and therefore $C_S$,
$C_M$, $C_L$ are identically distributed as well as the law of total expectation yields
\begin{align*}
  \E{C_n} &= \E{P_n} + \E{C_S} + \E{C_M} + \E{C_L}
  = \E{P_n} + 3 \E{C_M}\\&= \E{P_n} + 3\sum_{k=0}^{n-2}\P{M=k}\E{C_k}.
\end{align*}
Using \eqref{eq:distribution-M} leads to \eqref{eq:recurrence}. 

We recall how to solve recurrence~\eqref{eq:recurrence} using
generating functions. We follow 
Wild~\cite[\S~4.2.2]{Wild2013} who in turn follows
Hennequin~\cite{Hennequin:1991:analy}. The following lemma is contained in
slightly different notation in \cite{Wild2013}.

\begin{lemma}\label{le:integration}With the cost $C_n$ and $P_n$
  as above,
  $C(z)=\sum_{n\ge 0}\E{C_n}z^n$ and $P(z)=\sum_{n\ge 0}\E{P_n}z^n$, we have
  \begin{equation*}
    C(z)=(1-z)^3\int_{0}^z (1-t)^{-6}\int_{0}^t (1-s)^3P''(s)\,ds\,dt.
  \end{equation*}
\end{lemma}
For self-containedness, the proof is in Appendix~\ref{sec:proof-lemma-integration}.

\section{Partitioning Cost}\label{sec:part-costs}
In Section~\ref{sec:solve-recurrence} we saw that
in order to calculate the average number of
comparisons~$\E{C_n}$ of a dual-pivot quicksort algorithm
we need the expected partitioning cost~$\E{P_n}$
of the partitioning procedure used. 
The aim of this section is to determine $\E{P_n}$ 
for the partitioning procedure ``Count''
to be described below.

We use the set-up described at the beginning of
Section~\ref{sec:solve-recurrence}.  
For partitioning we use comparisons to \emph{classify}
the $n-2$ elements $a_2$, \dots, $a_{n-1}$ as \emph{small},
\emph{medium}, or \emph{large}. We will be using the term \emph{classification} 
for this central aspect of partitioning. Details of a partitioning 
procedure that concern how the classes are
represented or elements are moved around may and will be ignored. (Nonetheless,
in Appendix~\ref{sec:pseudocode}
we provide pseudocode for the considered classification strategies turned into 
dual-pivot quicksort algorithms.)
The cost $P_n$ depends on the concrete classification strategy used, 
the only relevant difference between
classification strategies being whether 
the next element to be classified is compared with
the smaller pivot~$p$ or the
larger pivot~$q$ first. This decision may
depend on the whole history of outcomes of previous comparisons.  (The
resulting abstract classification strategies may conveniently be
described as classification trees, see~\cite{AumullerD15}, but we do
not need this model here. See also Section~\ref{sec:part-strategies-opt}.)

Two comparisons are necessary for each medium element.
Furthermore, one comparison with $p$ is necessary for small and one
comparison with $q$ is necessary for large elements.
We call other comparisons occurring during classification \emph{additional comparisons}.
That means, an additional comparison arises when a small element is compared with $q$ first or a large element
is compared with $p$ first.

Next we describe the classification strategy ``Count''
from~\cite{AumullerD15}.
Let $s_t$ and $\ell_t$ denote the number of elements that have been
classified as small and large, respectively, in the first $t$ classification rounds. Set $s_0 = \ell_0 = 0$.

\begin{strategy}[``Count'']
  When classifying the $(t+1)$-st element, for $0\le t < n-2$, proceed as
  follows: If $s_t \geq \ell_t$, compare with $p$
  first, otherwise compare with $q$ first.
\end{strategy}

\begin{remark}
We argue informally that strategy ``Count'' is quite natural, 
referring to a standard method for parameter estimation from statistics. 
It is well known (and not hard to see) that,
as long as only comparisons are used, the following method for generating sequences 
to be sorted is equivalent to our probability model: 
Choose $a_1$, \dots, $a_n$ independently and
uniformly at random from $[0,1]$ (see, for example, \cite{Wild2013}).
One can imagine that first the pivots $\{a_1,a_n\}=\{p,q\}$ with $p<q$ 
are chosen uniformly at random from $[0,1]$. Elements $a_2$, \dots, $a_{n-1}$, chosen in the same way,
are classified in rounds $1$, \dots, $n-2$.
For $t<n-2$, the empirical frequencies $\frac{s_t}{s_t+\ell_t}$ and $\frac{\ell_t}{s_t+\ell_t}$ 
are the maximum-likelihood estimators for $p/(p+1-q)$ and $(1-q)/(p+1-q)$, 
which are the probabilities for the $(t+1)$-st element being small respectively large,
conditioned on it not being medium. 
So it is natural to bet that the next element to be considered will be small 
if and only if $\frac{s_t}{s_t + \ell_t} \ge\frac{\ell_t}{s_t + \ell_t}$.
This consideration even applies if $p$ and $q$ are fixed (and not chosen at random). 
Our analysis in Part~\partref{sec:count-is-optimal} 
shows that, when averaging over randomly chosen $p$ and $q$, 
strategy ``Count'' gives the smallest probability to pick the ``wrong'' pivot
among all strategies.
\end{remark}

In Appendix~\ref{sec:clairvoyant} we describe a closely related
strategy named ``Clairvoyant'', which assumes an oracle is given and requires slightly
fewer comparisons than ``Count''. In \cite{AumullerD15},
``Clairvoyant'' was used as an auxiliary means for analyzing ``Count''.

The goal of this section is to prove the following proposition.

\begin{proposition}\label{proposition:expected-partitioning-costs}
  Let $n\ge 2$. The expected partitioning cost of strategy ``Count'' is
    \begin{align}\label{eq:expected-partitioning-costs}
      \E{P_n} = \frac32n + \frac12\Hodd_n - \frac{19}{8}-\frac{3\iverson*{$n$ odd}}{8n}-\frac{\iverson*{$n$ even}}{8(n-1)}.
    \end{align}
  The corresponding generating function is
  \begin{multline}\label{eq:g-f-expected-partitioning-costs}
    P(z)=\sum_{n\ge 2} \E{P_n} z^n
    = \frac{3}{2(1-z)^2} + \frac{\artanh(z)}{2  {(1-z)}} \\
    -\frac{31z^2}{8(1-z)}-\frac{3+z}{8}\artanh(z) - \frac{3}{2}-\frac{25z}{8}.
  \end{multline}
  Asymptotically, we have
  \begin{equation*}
    \E{P_n} = \frac32 n + \frac14 \log n + \frac14\gamma + \frac{\log2}4 -
    \frac{19}{8} - \frac1{8n} - \frac1{12n^2}  + O\Bigl(\frac1{n^{3}}\Bigr)
  \end{equation*}
  as $n$ tends to infinity. 
\end{proposition}

We now translate the elements of the random permutation $(a_1,\ldots, a_n)$ that we are reading into a random lattice path $W$
 starting at the origin and using \emph{up-steps} $(1, +1)$ when a large element is
encountered, \emph{right-steps} $(1, 0)$ when a medium element is encountered and
\emph{down-steps} $(1, -1)$ when a
small element is encountered. Removing all $M$ right-steps from $W$ leads to a path $W'$ of length $n-2-M$
which only has up-steps and down-steps; see Figure~\ref{fig:qs:lattice:walk:ex}.

\begin{figure}
\hspace*{-1em}
\scalebox{0.85}{
		 \begin{tikzpicture}[yscale=1,xscale=0.8]
    \draw[->] (0, 0.75) to (0, 5.25);
    \draw[->] (0, 3) to (17.5, 3);
    \node at (1, 2.7) {$1$};
    \node at (17, 2.7) {$17$};
    \node (l2) at (17.75, 3) {$t$};
    \node (l1) at (-0, 5.5) {$\ell_t - s_t$};

    \foreach \x in {1, 2, 3, 4, 5, 6, 7, 8, 9, 10, 11, 12, 13, 14, 15, 16, 17}
        \draw (\x, 2.88) to (\x, 3.12);

    \foreach \x/\label in {1/-2, 2/-1, 3/\phantom{-}0, 4/\phantom{-}1, 5/\phantom{-}2}
    	\draw (-0.1, \x) to node[xshift=-0.5cm] {$\label$} (0.1, \x);
    \node[yshift=3cm, circle, fill, inner sep = 0, minimum size = 0.15cm] (p0) at (0,0) {};
    \node[yshift=3cm, circle, fill, inner sep = 0, minimum size = 0.15cm] (p1) at (1,1) {};
    \node[yshift=3cm, circle, fill, inner sep = 0, minimum size = 0.15cm] (p2) at (2,2) {};
    \node[yshift=3cm, circle, fill, inner sep = 0, minimum size = 0.15cm] (p3) at (3,1) {};
    \node[yshift=3cm, circle, fill, inner sep = 0, minimum size = 0.15cm] (p4) at (4,0) {};
    \node[yshift=3cm, circle, fill, inner sep = 0, minimum size = 0.15cm] (p5) at (5,-1) {};
    \node[yshift=3cm, circle, fill, inner sep = 0, minimum size = 0.15cm] (p6) at (6, 0) {};
    \node[yshift=3cm, circle, fill, inner sep = 0, minimum size = 0.15cm] (p7) at (7,-1) {};
    \node[yshift=3cm, circle, fill, inner sep = 0, minimum size = 0.15cm] (p8) at (8,-2) {};
    \node[yshift=3cm, circle, fill, inner sep = 0, minimum size = 0.15cm] (p9) at (9,-1) {};
    \node[yshift=3cm, circle, fill, inner sep = 0, minimum size = 0.15cm] (p10) at (10, 0) {};
    \node[yshift=3cm, circle, fill, inner sep = 0, minimum size = 0.15cm] (p11) at (11, 1) {};
    \node[yshift=3cm, circle, fill, inner sep = 0, minimum size = 0.15cm] (p12) at (12, 0) {};
    \node[yshift=3cm, circle, fill, inner sep = 0, minimum size = 0.15cm] (p13) at (13, 1) {};
    \node[yshift=3cm, circle, fill, inner sep = 0, minimum size = 0.15cm] (p14) at (14, 2) {};
    \node[yshift=3cm, circle, fill, inner sep = 0, minimum size = 0.15cm] (p15) at (15, 1) {};
    \node[yshift=3cm, circle, fill, inner sep = 0, minimum size = 0.15cm] (p16) at (16, 0) {};
    \node[yshift=3cm, circle, fill, inner sep = 0, minimum size = 0.15cm] (p17) at (17, -1) {};

    \draw (p0) to node[pos=0.5, draw, diamond, dotted, fill=white, inner sep = 0.05cm] {${\lambda}$} (p1);

    \draw (p1) to node[pos=0.5, fill=white, dotted, inner sep=0.05cm] {$\lambda$}  (p2);
    \draw (p2) to node[pos=0.5, draw, diamond, fill=white, dotted, inner sep=0.05cm] {${\sigma}$}  (p3) ;
    \draw (p3) to node[pos=0.5, draw, diamond, fill=white, dotted, inner sep=0.05cm] {${\sigma}$}  (p4) ;
    \draw (p4) to node[pos=0.5, fill=white, dotted, inner sep=0.05cm] {$\sigma$}  (p5) ;
    \draw (p5) to node[pos=0.5, draw, diamond, fill=white, dotted, inner sep=0.05cm] {${\lambda}$}  (p6) ;
    \draw (p6) to node[pos=0.5, fill=white, dotted, inner sep=0.05cm] {$\sigma$}  (p7) ;
    \draw (p7) to node[pos=0.5, fill=white, dotted, inner sep=0.05cm] {$\sigma$}  (p8) ;
    \draw (p8) to node[pos=0.5, draw, diamond, fill=white, dotted, inner sep=0.05cm] {${\lambda}$}  (p9) ;
    \draw (p9) to node[pos=0.5, draw, diamond, fill=white, dotted, inner sep=0.05cm] {${\lambda}$}  (p10) ;
    \draw (p10) to node[pos=0.5, draw, diamond, fill=white, dotted, inner sep=0.05cm] {${\lambda}$}  (p11) ;
    \draw (p11) to node[pos=0.5, draw, diamond, fill=white, dotted, inner sep=0.05cm] {${\sigma}$}  (p12) ;
    \draw (p12) to node[pos=0.5,  draw, diamond, fill=white, dotted, inner sep=0.05cm] {${{\lambda}}$}  (p13) ;
    \draw (p13) to node[pos=0.5, fill=white,  inner sep=0.05cm] {$\lambda$}  (p14) ;
    \draw (p14) to node[pos=0.5, draw, diamond, fill=white, dotted, inner sep=0.05cm] {${\sigma}$}  (p15) ;
    \draw (p15) to node[pos=0.5, draw, diamond, fill=white, dotted, inner sep=0.05cm] {${\sigma}$}  (p16) ;
    \draw (p16) to node[pos=0.5, fill=white, dotted, inner sep=0.05cm] {${\sigma}$}  (p17);
\end{tikzpicture}}

    \caption{Example run of strategy ``Count'' on input $b=(\lambda$, $\lambda$, $\sigma$, $\sigma$, $\sigma$, $\lambda$, $\sigma$, $\sigma$, $\lambda$, $\lambda$, $\lambda$, $\sigma$, $\lambda$, $\lambda$, $\sigma$, $\sigma$, $\sigma)$. The $s=9$ small elements 
		in $b$ are represented 
		by the symbol $\sigma$, the $\ell=8$ large elements by $\lambda$ (no medium elements $\mu$ present). 
		The horizontal axis shows the input position $t$, the vertical axis shows the difference $\ell_t-s_t$ of large and small elements in the initial segment $b_{\le t}=(b_1,\ldots,b_t)$ of the input. 
  A diamond marks elements where an additional comparison occurs using strategy ``Count''.
  It makes eleven additional comparisons on this input.}
\label{fig:qs:lattice:walk:ex}
\end{figure}
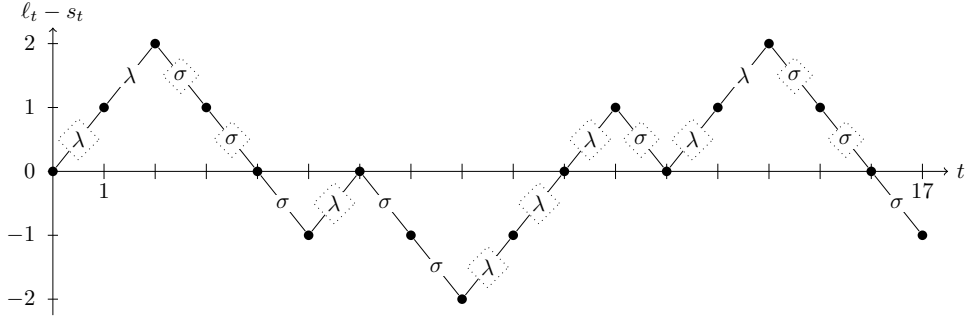

It now turns out that $W'$ is a path which is distributed exactly as the random
paths studied in Part~\partref{sec:lattice-paths}.

\begin{lemma}Let $0\le m\le n-2$. We have
  \begin{equation}\label{eq:equidistribution-ending-points-w-prime}
    \P{W'\text{ ends at }(n-m-2, d)\mid M=m} = \frac{1}{n-m-1}
  \end{equation}
  for all $d\equiv n-m-2\pmod 2$ and $\abs{d}\le n-m-2$.

  For such $m$ and $d$, all paths $W'$ ending in $(n-m-2, d)$ are equally likely.
\end{lemma}
\begin{proof}
 As there are $L$ up-steps and $S$ down-steps, the lattice path $W$ ends at $(n-2,
 L-S)$. Thus $W'$ ends on $(n-2-M, L-S)$.

 Once $S=s$, $M=m$ and $L=\ell$ are fixed, there are ${\binom{n-2}{s, m, \ell}}$
 paths with that numbers of up-steps, right-steps and down-steps,
 respectively. All of them are equally likely because they only differ by the
 order of the steps and all orders are equally likely by our probability model
 on the input list. Thus,
 \begin{equation*}
   \P{W=w \mid (S, M, L)=(s,m,\ell)} = \frac1{\binom{n-2}{s, m, \ell}}
 \end{equation*}
 for all paths $w$ with $\ell$ up-steps, $m$ right-steps and $s$ down-steps.

 We have
 \begin{equation*}
   \P{W'\text{ ends at }(n-m-2, d)\mid M=m}=\frac{\P{\text{$L-S=d$ and $L+S=n-m-2$}}}{\P{M=m}}.
 \end{equation*}
 The expression in the numerator uniquely determines the pair $(S, L)$. As each
 of these pairs with $S+m+L=n-2$ is equally likely, the numerator cannot
 depend on $d$ if it is positive at all. It is easily seen that the probability
 is positive if and only if $d\equiv n-m-2\pmod 2$ and $\abs{d}\le n-m-2$. There
 are $n-m-1$ choices for $d$, thus we
 obtain~\eqref{eq:equidistribution-ending-points-w-prime}.

 Each lattice path $w'$ of length $n-m-2$
 corresponds to $\binom{n-2}{m}$ paths $w$ with up-, right- and down-steps of
 length $n-2$. Thus all paths $w'$ ending at $(n-m-2, d)$ are equally likely.
\end{proof}

We are now able to describe the partitioning cost.
\begin{lemma}
  For $n\ge 2$, the partitioning cost $P_n$ is
  \begin{equation}\label{eq:P_n}
    P_n=1+\frac{3}2 (n-2) + \frac{1}{2} M + Z_{n-2-M}^{\nearrow}-\frac12\abs{L-S}
  \end{equation}
  where $Z_{n-2-M}^{\nearrow}$ denotes the number of up-from-zero situations (cf.\ Section~\ref{sec:prob}) of
  the random path $W'$.   
\end{lemma}
\begin{proof}
  The first summand $1$ corresponds to the comparison between the two pivot elements.
  We note that the number of up-from-zero situations of $W$ equals
  $Z_{n-2-M}^{\nearrow}$ by the construction of $W'$, because omitting the
  right-steps clearly does not change the number of up-from-zero situations.

  A right-step (medium element) always has
  partitioning cost of $2=\frac32+\frac12$. We split the cost of $2$ induced by an
  up-from-zero situation (compare with smaller pivot first, but read a large element)
  into $1+1$, the latter being taken into account by
  the summand  $Z_{n-2-M}^{\nearrow}$. Thus, for the remainder of this proof,
  we only have to consider a remaining cost of $1$ for those steps. A step away
  from the axis (correct pivot first) then costs $1=3/2-1/2$,
  a step towards the axis (wrong pivot first) costs $2=3/2+1/2$. The sum of the
  numbers of these steps is
  $n-2-M$, their difference is $\abs{L-S}$. This amounts to a contribution of $\frac32(n-M-2)-\frac12\abs{L-S}$.
\end{proof}

In order to compute the generating function of $\E{P_n}$, we have to compute the generating function of $\E[normal]{Z_{n-2-M_n}^\nearrow}$.

\begin{lemma}\label{lemma:expectation-Z-n-2-M} The generating function of
  $\E[normal]{Z_{n-2-M_n}^\nearrow}$ is
  \begin{equation}\label{eq:G-F-E-Z-n-2-M-n-e}
    \sum_{n\ge 2}\E[normal]{Z_{n-2-M_n}^\nearrow} z^{n} =\frac{\artanh(z)}{2  {(1-z)}}-
    \frac{z^2}{8  {(1-z)}} - \frac{3z + 5}{8}  \artanh(z) + \frac{1}{8}  z.
  \end{equation}
\end{lemma}

\begin{proof}
  The law of total expectation yields
  \begin{equation*}
    \E[normal]{Z_{n-2-M_n}^\nearrow} = \sum_{m=0}^{n-2}
    \P{M_n=m} \E[empty]{Z_{n-2-m}^\nearrow}.
  \end{equation*}

  From \eqref{eq:distribution-M} and
  Corollary~\ref{cor:exp-up-from-zero}, we immediately 
  get that
  \begin{equation}\label{eq:E-Z-n-M-2-rewritten}
  \begin{aligned}
    \E[normal]{Z_{n-2-M_n}^\nearrow} &= \frac{1}{2\binom{n}{2}}
    \sum_{m=0}^{n-2} (n-m-1) \Hodd_{n-2-m}=
    \frac{1}{2\binom{n}{2}} \sum_{m=0}^{n-2} (m+1)\Hodd_{m}\\&= 
    \frac{1}{2\binom{n}{2}}\sum_{m=0}^{n-2} \sum_{k=1}^m\iverson*{$k$ odd}\frac{ m+1}k.
  \end{aligned}
  \end{equation}

  We now consider the generating function
  \begin{equation*}
    G(z) \coloneqq \sum_{n\ge 2}\E[normal]{Z_{n-2-M_n}^\nearrow}z^n.
  \end{equation*}
  By \eqref{eq:E-Z-n-M-2-rewritten}, we have
  \begin{equation*}
    G''(z) = \sum_{n\ge 2} \sum_{m=0}^{n-2}\sum_{k=1}^m \frac{\iverson*{$k$
        odd}}{k} (m+1)z^{n-2}.
  \end{equation*}
  Exchanging the order of summation and replacing $n-2$ by $n$ yields
  \begin{align*}
    G''(z) &= \sum_{k\ge 1} \frac{\iverson*{$k$
        odd}}{k} \sum_{m\ge k}(m+1)\sum_{n\ge m}z^n\\
    &=\frac1{1-z}\sum_{k\ge 1} \frac{\iverson*{$k$
        odd}}{k} \sum_{m\ge k}(m+1) z^m\\
    &=\frac1{1-z} \sum_{k\ge 1}\frac{\iverson*{$k$
        odd}}{k}\Bigl( \frac{kz^k}{1-z}+ \frac{z^{k}}{(1-z)^2}\Bigr)\\
    &= \frac1{(1-z)^2}\sum_{k\ge 1}\iverson*{$k$
        odd}z^k + \frac{1}{(1-z)^3}\sum_{k\ge 1}\frac{\iverson*{$k$
        odd}}{k}z^k \\
    &=\frac{z}{(1-z)^3(1+z)} + \frac1{(1-z)^3}\artanh(z).
  \end{align*}
  Integrating twice yields \eqref{eq:G-F-E-Z-n-2-M-n-e}. Note that the 
  summand $1/8z$ ensures that $G(z)=O(z^2)$.
\end{proof}

\begin{proof}[Proof of Proposition~\ref{proposition:expected-partitioning-costs}]
  Taking expectations in \eqref{eq:P_n} and summing up the contributions of
  \eqref{eq:expectation-M}, \eqref{eq:expectation-difference} and Lemma~\ref{lemma:expectation-Z-n-2-M}
  yields \eqref{eq:g-f-expected-partitioning-costs}.

  For deriving \eqref{eq:expected-partitioning-costs}, we use the expansions
  \begin{equation}\label{eq:artanh}
    \sum_{n> k}\frac{\iverson*{$n-k$ odd}}{n-k}z^{n}=
    z^k\artanh(z).
  \end{equation}
  valid for $k\in\Z$.
  The asymptotic expansion follows from Lemma~\ref{lem:harmonic-asy}.
\end{proof}

\section{Main Results and their Asymptotic Aspects}
\label{sec:costs-main-asy}

In this section we give precise formulations and proofs of our main results. 
We use the partitioning cost from the previous section to
calculate the expected number of key comparisons of the dual-pivot quicksort variant
obtained by using classification strategy ``Count''.
We call this sorting algorithm ``Count'' again.

\begin{theorem}\label{thm:count:cost}
  For $n\ge 4$, the average number of comparisons in the comparison-optimal dual-pivot
  quicksort algorithm ``Count''
  when sorting a list of $n$ elements is
  \begin{multline*}
    \E{C_n} = 
    \frac{9}{5}nH_n - \frac{1}{5}n\Halt_n -\frac{89}{25}n + \frac{67}{40}H_n-\frac{3}{40}\Halt_n-\frac{83}{800}+\frac{(-1)^n}{10}\\
    - \frac{\iverson*{$n$ even}}{320}\Bigl(\frac{1}{n-3}+\frac{3}{n-1}\Bigr)+ \frac{\iverson*{$n$ odd}}{320}\Bigl(\frac{3}{n-2}+\frac{1}{n}\Bigr).
  \end{multline*}
\end{theorem}

The sequence $(n! \E{C_n})_{n\geq0}$ is \href{https://oeis.org/A288965}{A288965}
in The On-Line Encyclopedia of Integer Sequences~\cite{OEIS:2017}.
From the above exact result it is not hard to determine the first
few terms of the asymptotic behavior. This is formulated as the following
corollary.

\begin{corollary}\label{cor:count:cost:asy}
  The average number of comparisons in the optimal dual-pivot
  quicksort algorithm ``Count''
  when sorting a list of $n$ elements is
  \begin{equation*}
    \E{C_n} = 
    \frac95n \log n + A n + B \log n + C
    + \frac{D}{n} + \frac{E}{n^2} + \frac{F\iverson*{$n$ even} + G}{n^3}
    + \Oh[Big]{\frac{1}{n^4}}
  \end{equation*}
  with
  \begin{align*}
    A &= \frac95\gamma
    + \frac{1}{5} \log 2
    - \frac{89}{25}
    = -2.3823823670652\dots, &
    B &= \frac{67}{40}
    = 1.675, \\
    C &= \frac{67}{40}\gamma
    + \frac{3}{40}\log 2
    + \frac{637}{800}
    = 1.81507227725206\dots, &
    D &= \frac{11}{16} = 0.6875, \\
    E &= - \frac{67}{480} = -0.139583333333333\dots, \\
    F &= - \frac{1}{8} = -0.125, &
    G &= \frac{31}{400} = 0.0775,
  \end{align*}
  asymptotically as $n$ tends to infinity.
\end{corollary}

\begin{corollary}\label{cor:count:better:than:classical}
  For each input size of a list of elements, algorithm ``Count''
  uses at most as many comparisons on average as classical quicksort (with one
  pivot).
\end{corollary}

Before continuing, let us
make a remark on the (non-)influence of the parity of~$n$. It is noteworthy
that in Corollary~\ref{cor:count:cost:asy} no such
influence is visible in the first six terms (down to $1/n^2$); only
from $1/n^3$ on the parity of $n$ appears. This is somewhat unexpected, 
since a term $(-1)^n/10$ appears in Theorem~\ref{thm:count:cost}.

\begin{proof}[Proof of Theorem~\ref{thm:count:cost}]
  The partitioning cost and the generating function of strategy
  ``Count'' are stated in
  Proposition~\ref{proposition:expected-partitioning-costs}.

  We calculate the comparison cost from the partitioning cost by means
  of Lemma~\ref{le:integration} and obtain
  \begin{multline*}
    C(z) = 
    - \frac{8\log(1-z)}{5(1-z)^2}+\frac{2\artanh(z)}{5 (1-z)^2}
    - \frac{44}{25  (1-z)^2}
    - \frac{\artanh(z)}{4 (1-z)}
    + \frac{281}{160 (1-z)}\\
    + \frac{(1-z)^3}{320}\artanh(z) + \frac{1}{150} z^{3}
    - \frac{27}{1600} z^{2} + \frac{17}{1600} z + \frac{3}{800}.
  \end{multline*}
  Taking into account that $\artanh(z)=(\log(1+z)-\log(1-z))/2$,
  \begin{align*}
    \sum_{m\ge 1}\Halt_m z^m&=-\frac{\log(1+z)}{(1-z)},\\
    \sum_{m\ge 1} H_m z^m&=-\frac{\log(1-z)}{(1-z)},\\
    \sum_{m\ge 1}m\Halt_m z^{m}&=z \Bigl(-\frac{\log(1+z)}{(1-z)}\Bigr)'\\
    &=- \frac{\log(1 + z)}{{(1-z)}^{2}}+\frac{\log(1 + z)}{1-z} + \frac{1}{2  {(1 + z)}} - \frac{1}{2  {(1-z)}}, \\
    \sum_{m\ge 1}m H_m z^{m}&=z \Bigl(-\frac{\log(1-z)}{(1-z)}\Bigr)'\\
    &= - \frac{\log(1-z)}{{(1-z)}^{2}}+ \frac{1}{{(1-z)}^{2}}+\frac{\log(1-z)}{1-z} - \frac{1}{1-z} ,
  \end{align*}
  as well as \eqref{eq:artanh}, we obtain the result.
\end{proof}

\begin{proof}[Proof of Corollary~\ref{cor:count:cost:asy}]
  Insert the expansions of Lemma~\ref{lem:harmonic-asy} into the explicit
  representations of Theorem~\ref{thm:count:cost}.
\end{proof}

\begin{proof}[Proof of Corollary~\ref{cor:count:better:than:classical}]
  We must compare the cost $c_n = 2(n+1)H_n - 4n$ of classical
  quicksort---this formula can be derived as described in
  Knuth~\cite[p.~120]{Knuth:1998:Art:3}---with the exact formula for
  $\E{C_n}$ from Theorem~\ref{thm:count:cost}.

	It is easily seen by looking at the algorithms that $c_n = \E{C_n}$ for $n\in\{1,2,3\}$.
	For $n\in\{4,\dots,10\}$ one evaluates the formula in Theorem~\ref{thm:count:cost} 
	and compares with $c_n$. (Or one compares the first ten terms in \href{https://oeis.org/A288964}{A288964}
	and \href{https://oeis.org/A288965}{A288965}, in~\cite{OEIS:2017}.)

	For $n\ge11$ we argue as follows. We have
        \begin{equation*}
          c_n-\E{C_n}=\frac15n\Bigl(H_n+\Halt_n-\frac{11}{5}\Bigr) + \frac14H_n +
          \frac3{40}(H_n+\Halt_n)+s
        \end{equation*}
        where $s\ge -1/100$ for $n\ge 10$.  Note that
        $H_n+\Halt_n= H_{\floor{n/2}}$. Thus the difference is clearly
        positive as soon as $H_{\floor{n/2}}\ge \frac{11}{5}$ which is the case for
        $n\ge 10$.
\end{proof}

\part{Optimality of the Strategy ``Count''}
\label{sec:count-is-optimal}

The aim of this section is to show the optimality of the partitioning
strategy ``Count'' in the sense that it minimizes the expected number
of key comparisons among all possible partitioning strategies. This is
formulated precisely as Theorem~\ref{thm:count-optimal} at the end of
this part.

\section{Input Sequences}\label{subsec:input-seq}

Let $n$ be given. For a random permutation $(a_1, \ldots, a_n)$, we use the random variables $S$, $M$ and
$L$ of Section~\ref{sec:solve-recurrence} to create a new random
variable~$B$ whose values are sequences of length~$n-2$ of $S$ letters
$\sigma$, $M$ letters $\mu$  and $L$ letters $\lambda$ representing the small,
medium and large elements in $(a_2,\ldots, a_{n-1})$, respectively.

We identify $B$ with the random lattice path $W\in\N_0^3$ starting in the origin where $\sigma$, $\mu$ and
$\lambda$ correspond to steps $e_1$, $e_2$ and $e_3$, respectively. By
Lemma~\ref{lemma:S-M-L-distribution}, $(S, M, L)$ and therefore the end point
of $W$ are uniformly distributed in $\Omega'$ (cf.\@ \eqref{eq:def-Omega}). For each $(s, m,
\ell)\in\Omega'$, each such lattice path ending in $(s, m, \ell)$ is equally
likely. Therefore, Lemma~\ref{lemma:k-dimensional-lattice-paths} is applicable.

We will use the notation $\abs{b}$ for the length of the sequence~$b$
(above $\abs{b}=n-2)$ and $\abs{b}_\gamma$ for the number of $\gamma$'s
($\gamma\in\set{\sigma,\mu,\lambda}$) in the sequence $b$ (above, for example,
$\abs{b}_\mu=m$). For $0\le t\le n$, we write $b_{\le t}$ for the initial
segment of length $t$ of $b$ and $b_t$ for the $t$-th symbol of $b$.

Lemma~\ref{lemma:k-dimensional-lattice-paths} implies that
\begin{equation}\label{eq:conditional-probability-optimality-proof}
  \P{ B_{t+1}= \gamma\mid B_{\le t}=b_{\le t}} = \frac{\abs{b_{\le t}}_\gamma+1}{t+3}
\end{equation}
for all sequences $b$, $0\le t<n$ and $\gamma\in\{\sigma,\mu,\lambda\}$.

\section{Partitioning Strategies}
\label{sec:part-strategies-opt}

Since we want to compare ``Count'' with arbitrary partitioning strategies,
we must say what such a strategy is in general. 

We start with a remark. Partitioning strategies carry out comparisons. 
As observed in~\cite{AumullerD15}, it does not help for saving comparisons to 
postpone comparing an element with the second pivot if the first 
comparison was not sufficient to classify the element. Thus we may assume
that a strategy always classifies an element completely, either by one
or, if necessary, by two comparisons. Thus, strategies we consider classify
$n-2$ elements, one after the other. 

A second remark concerns the order in which elements are looked at. 
While a partitioning strategy can decide in every round which element
of the input is looked at next, this decision does not matter for the 
overall cost (expected number of key comparisons). The reason for this is that the elements not tested
up to this point are in random order, so it is irrelevant which of
them the strategy chooses for treating next.  We use this observation
for justifying the assumption that the $n-2$ elements of a sequence
are just read and processed in the order given in the sequence.

Now we can easily describe what a partitioning strategy is. For
each sequence of $\sigma$'s, $\mu$'s, and $\lambda$'s of length smaller than $n-2$
a strategy specifies whether the next element is to be compared
with the smaller pivot $p$ or with the larger pivot $q$ first.
Let $\set{\sigma, \mu, \lambda}^{<n-2}$ denote the set of all
sequences over $\set{\sigma, \mu, \lambda}$ of length smaller than $n-2$.%
\footnote{Of course, exactly the same information is contained 
in a classification tree as in \cite{AumullerD15}, if 
the labels of the elements are ignored.
However, the function notation is more convenient here.}

\begin{definition}
A partitioning strategy for inputs of length $n$ is a function
\begin{equation*}
\strgy \colon \set{\sigma, \mu, \lambda}^{<n-2} \to \set{p,q},
\end{equation*}
where $\f{\strgy}{\tau}=p$ means that after having seen the initial segment $\tau$ 
of a sequence the 
next element is compared with $p$ first, similarly for $\strgy(\tau)=q$.
\end{definition}

\begin{example*}
  Strategy ``Count'' is given by the function
\begin{equation*}
  \f{\strgy^{\mathrm{ct}}}{\tau} =
  \begin{cases}
    p & \text{if $\abs{\tau}_{\sigma} \ge \abs{\tau}_{\lambda}$},\\
    q & \text{otherwise.}
  \end{cases}
\end{equation*}
\end{example*}

\begin{remark}
We know that medium elements, i.e., the $\mu$'s in these sequences,
do not influence the additional comparison count. So one might be tempted
to model classification strategies without taking medium elements into account. 
However, this leads into difficulties, since medium elements encountered underway
may influence the decisions made by partitioning procedures
(for example, after seeing one $\sigma$, a strategy might opt for $p$,
but after seeing one $\sigma$ and three $\mu$'s, opt for $q$).
These differences would be hard to take into account if we 
left out the middle elements from the model. 

Note that implicitly the step number is an argument of $\strgy$, 
so any reaction to the number of steps made so far can also built into 
$\strgy$.
Nothing would change if it depended on any other parameters or was randomized.
\end{remark}

\section{Optimality of the Strategy ``Count''}\label{subsec:add-cost}

Assume $b=(b_1,\dots,b_{n-2})\in\set{\sigma, \mu, \lambda}^{n-2}$ is a
sequence.  Let $b_{\le t}$ denote the sequence $(b_1,\dots,b_t)$.
We use the terminology \emph{additional cost} of Section~\ref{sec:part-costs}.
A strategy $\strgy$ incurs additional cost $1$ in step $t+1$ of the
partitioning of $b$ if and only if
\begin{itemize}
\item $\strgy(b_{\le t})=p$ and $b_{t+1}=\lambda$, or
\item $\strgy(b_{\le t})=q$ and $b_{t+1}=\sigma$.
\end{itemize}
For a strategy~$\strgy$ and a random sequence~$B$ (as in Section~\ref{subsec:input-seq}), let the random variable $A^{\strgy}$ be the total number of additional comparisons caused by the sequence~$B$.
We are interested in $\E{A^{\strgy}}$ and show the proposition below,
which implies the main result Theorem~\ref{thm:count-optimal}.

\begin{proposition}\label{pro:count-optimal}
Let $\mathrm{ct}$ be strategy ``Count'', and let $\strgy$ be an arbitrary strategy. 
Then \emph{$\E{A^{\mathrm{ct}}\text{}} \le \E{A^{\strgy}}$}.
\end{proposition}
\begin{proof}
Note that for an arbitrary strategy $\strgy$ we have 
\begin{multline*}
  \E{A^{\strgy}} = \\
  \sum_{\tau \in \set{\sigma, \mu, \lambda}^{<n-2}}
  \P{\text{$\strgy$ incurs additional cost $1$ in step $\abs{\tau}+1$}\mid B_{\le\abs{\tau}}=\tau} \P{B_{\le\abs{\tau}}=\tau}
\nonumber
\end{multline*}
by linearity of expectation and the law of total expectation.
In view of this formula all we have to do is to show
that
\begin{multline}\label{eq:core-optimality}
\P{\text{$\mathrm{ct}$ incurs additional cost 1 in step $\abs{\tau}+1$}\mid B_{\le\abs{\tau}}=\tau}\\
\le \P{\text{$\strgy$ incurs additional cost 1 in step $\abs{\tau}+1$}\mid B_{\le\abs{\tau}}=\tau}
\end{multline}
for all sequences $\tau \in \set{\sigma, \mu, \lambda}^{<n-2}$.

So assume $1\le t < n-2$ 
and $\tau\in\set{\sigma, \mu, \lambda}^t$ are given.
We work under the assumption that $B_{\le t}= \tau$.

The first case we consider is that $\abs{\tau}_\sigma \ge \abs{\tau}_\lambda$. 
Then ``Count'' chooses to compare with pivot $p$ first, 
so the probability that it incurs an additional comparison is 
\begin{equation}\label{eq:count-compare}
 \frac{\abs{\tau}_\lambda + 1}{t+3}
\end{equation}
by \eqref{eq:conditional-probability-optimality-proof}.
Now consider an arbitrary strategy $\strgy$. The probability that $\strgy$ incurs an 
extra comparison at step $t+1$ is
\begin{equation}\label{eq:arbitrary-compare}
\iverson{\strgy(\tau)=p}\cdot \frac{\abs{\tau}_\lambda + 1}{t+3} + 
\iverson{\strgy(\tau)=q}\cdot \frac{\abs{\tau}_\sigma + 1}{t+3},
\end{equation}
which is at least as big as \eqref{eq:count-compare} by the assumption that $\abs{\tau}_\sigma \ge \abs{\tau}_\lambda$. 
(Note that this argument would apply \emph{mutatis mutandis} if 
strategy $\strgy$ were randomized and chose
between $p$ and $q$ by a random experiment.)

The second case $\abs{\tau}_\sigma < \abs{\tau}_\lambda$ is similar.
This finishes the proof of Proposition~\ref{pro:count-optimal}. 
\end{proof}
We remark that the proof of Proposition~\ref{pro:count-optimal} shows that, in fact, each
optimal strategy must choose the same pivot as ``Count'' for the first
comparison in each step with $\abs{\tau}_\sigma \neq \abs{\tau}_\lambda$.

\begin{theorem}\label{thm:count-optimal}
  Let $n\geq2$. The expected number of key comparisons of any
  partitioning strategy (according to
  Section~\ref{sec:part-strategies-opt}) when classifying a list of
  $n$ elements is at least $\E{\Pct_n}$ of
  Proposition~\ref{proposition:expected-partitioning-costs}
  and this bound is sharp by the partitioning strategy ``Count''.
\end{theorem}

Reformulated, the partitioning strategy ``Count'' minimizes the
average number of key comparisons.

\begin{proof}[Proof of Theorem~\ref{thm:count-optimal}]
  Each strategy needs the same number of necessary key
  comparisons. Proposition~\ref{pro:count-optimal} deals with the
  additional number of key comparisons. The result follows.
\end{proof}

\begin{remark}A well-known variant of choosing the pivot(s) in classical 
quicksort and in multi-pivot quicksort is \emph{pivot sampling}. 
In a randomly permuted input $(a_1,\dots,a_n)$ one chooses
$k$ elements (e.g., $a_1$, \dots, $a_k$) and chooses the pivot(s) 
as elements of suitable ranks in this set.
The method ``median-of-three'' of classical quicksort uses $k=3$ and takes the median as pivot.
In the actual implementation of the YBB algorithm the two pivots are 
the elements of rank 2 and 4 in a sample of size $k=5$.
Pivot sampling for dual pivot quicksort and multi-pivot quicksort
was, for example, studied by Hennequin~\cite{Hennequin:1991:analy} and by Wild~\cite{Wild:2016:Phd-dual-pivot}.
We remark here that strategy ``Count'' is optimal even in this slightly more 
general situation. With $k$ samples $p_1 < \dots < p_k$
(the elements of $\{a_1,\dots,a_k\}$ in sorted order), $p_0=-\infty$, and $p_{k+1}=\infty$,
the set $\{a_{k+1},\dots,a_n\}$
of the remaining elements is split into $k+1$ classes $A_0$, \dots, $A_{k}$, with 
elements between $p_i$ and $p_{i+1}$ in class $A_i$. 
If the pivots are $p=p_j$ and $q=p_h$, where $j<h$, then elements in $\{p_1,\dots,p_{j-1}\}\cup A_0\cup\cdots\cup A_{j-1}$
are ``small'', elements in $\{p_{j+1},\dots,p_{h-1}\}\cup A_{j}\cup\cdots\cup A_{h-1}$ are ``medium'', 
and elements in $\{p_{h+1},\dots,p_{k}\}\cup A_{h}\cup\cdots\cup A_k$ are ``large''.
When running the algorithm, we consider sequences $\tau$ in $\{0,1,\dots,k\}^{n-k}$ to indicate the classes of 
the elements seen and classified so far. Strategy ``Count'' is the same as before,
except that in the counts we include sampled elements that were not chosen as pivots. 
Consider round $t+1$.
The classes of the elements from $\{a_{k+1},\dots,a_n\}$ seen so far determine a sequence $\tau\in\{0,1,\dots,k\}^t$.
The algorithm then knows $s_t=|\tau|_0+\dots+|\tau|_{j-1}+j-1$ and $\ell_t=|\tau|_{h}+\dots+|\tau|_k+k-h$,
the numbers of small and large elements seen so far.
We may apply Lemma~\ref{lemma:k-dimensional-lattice-paths} just as before to see that the
probability that the next element is in class $A_i$ is $\frac{|\tau|_i+1}{t+k+1}$,
and hence that the probability that the next element is large is
\begin{equation*}
  \frac{\sum_{h \le i\le k}(|\tau|_i+1)}{t+k+1} = \frac{\ell_t+1}{t+k+1} .
\end{equation*}
Similarly, the probability that the next element is small is $\frac{s_t+1}{t+k+1}$.
Apart from the larger denominator, these are the same formulas as in the proof of Proposition~\ref{pro:count-optimal},
and so the proof that strategy ``Count'' minimizes the probability for an additional comparison
works as before.
\end{remark}

\section{Future Work}\label{subsec:future:work}

We derived the exact expected number of comparison in the case that the two pivots 
are chosen at random from the input and
dual-pivot quicksort is used to sort the input without stopping the recursion early.
Pivot sampling, as discussed in a remark in Section~\ref{subsec:add-cost},
is a variant of quicksort with two or more pivots that is relevant in practice.
The leading term of the expected number of comparisons is known
\cite{AumullerD15,NebelWM16}, but no analysis of lower order terms has been
conducted.  
Additionally, inputs of
size at most $M$, for $M$ being a small integer, are usually sorted with insertion sort. This
approach is analyzed for the YBB algorithm in \cite{WildNN15}. It would be interesting to see
how their techniques can be applied in our analysis.

Another line of research is the obvious generalization to quicksort
algorithms using more than two pivot elements
\cite{AumullerDK16}. Since the original version of this
article was submitted, some progress has been made by two of the authors.

\bibliographystyle{amsplainurl}
\bibliography{lit}

\clearpage
\appendix

\section{Proof of Lemma~\ref{le:integration}}\label{sec:proof-lemma-integration}

\begin{proof}
Multiplying \eqref{eq:recurrence} by $n(n-1)z^{n-2}$ and summing over all $n\ge
2$ yields
\begin{multline*}
  \sum_{n\ge 2}n(n-1)\E{C_n}z^{n-2}\\=\sum_{n\ge 2}n(n-1)\E{P_n}z^{n-2} + 6
  \sum_{n\ge 1}\sum_{k=0}^{n-1}(n-1-k)z^{n-k-2} \E{C_k}z^k.
\end{multline*}
Note that the range of the summations has been extended without any
consequences because of $\E{C_0}=0$. We replace $n-1$ by $n$ in the double sum
and write it as a product of
two generating functions:
\begin{multline*}
  \sum_{n\ge 1}\sum_{k=0}^{n-1}(n-1-k)z^{n-k-2} \E{C_k}z^k = 
  \sum_{n\ge 0}\sum_{k=0}^{n}(n-k)z^{n-k-1} \E{C_k}z^k \\= \biggl(\sum_{n\ge
    0}nz^{n-1}\biggr)C(z)=\biggl(\sum_{n\ge 0}z^n\biggr)'
  C(z)=\Bigl(\frac1{1-z}\Bigr)' C(z)=\frac{C(z)}{(1-z)^2}.
\end{multline*}
Thus we obtain
\begin{equation*}
  C''(z)=P''(z)+\frac{6}{(1-z)^2} C(z)
\end{equation*}
or, equivalently,
\begin{equation*}
  (1-z)^2C''(z)- 6 C(z)=(1-z)^2P''(z).
\end{equation*}
Setting $(\theta f)(z)=(1-z)f'(z)$ for a function $f$, this can be rewritten as
\begin{equation*}
  ((\theta^2+\theta-6)C)(z) = (1-z)^2P''(z).
\end{equation*}
Factoring $\theta^2+\theta-6$ as $(\theta-2)(\theta+3)$ and setting
$D=(\theta+3)C$, we first have to solve
\begin{equation*}
  ((\theta-2)D)(z)=(1-z)^2P''(z),
\end{equation*}
i.\,e.,
\begin{equation*}
  (1-z)D'(z)-2D(z)=(1-z)^2 P''(z).
\end{equation*}
Multiplication by $(1-z)$ yields
\begin{equation*}
  \bigl((1-z)^2 D(z)\bigr)'=(1-z)^3 P''(z).
\end{equation*}
Integration and the fact that $D(0)=C'(0)+3C(0)=\E{C_1+3C_0}=0$ yields
\begin{equation*}
  D(z)=\frac1{(1-z)^2}\int_0^z (1-s)^3P''(s)\,ds.
\end{equation*}
We still have to solve
\begin{equation*}
  (1-z)C'(z)+3C(z)=D(z).
\end{equation*}
We multiply by $(1-z)^{-4}$ and obtain
\begin{equation*}
  \bigl((1-z)^{-3}C(z)\bigr)'=(1-z)^{-4}D(z).
\end{equation*}
As $C(0)=0$, we obtain
\begin{equation*}
  C(z)=(1-z)^3\int_{0}^z (1-t)^{-4} D(t)\, dt.
\end{equation*}
\end{proof}

\section{Strategy Clairvoyant}
\label{sec:clairvoyant}

Historically, the idea of the strategy ``Count'' arose from
the non-algorithmic strategy ``Clairvoyant'', which is described in this section; see also \cite{AumullerD15}. It will turn
out that its
cost is only marginally smaller than 
that of strategy ``Count''.

Let $s_t$ and $\ell_t$ denote the number of elements that have been
classified as small and large, respectively, in the first $t$ classification rounds. Set $s_0 = \ell_0 = 0$ and denote the total number of small and large elements by $s$ and $\ell$ respectively.

\begin{strategy}[``Clairvoyant'']
    When classifying the $(t + 1)$-st
  element, for $0\le t < n-2$, proceed as follows: If $s - s_{t}
  \geq \ell - \ell_{t}$, compare with $p$ first, otherwise
  compare with $q$ first.
\end{strategy}

Note that the strategy ``Clairvoyant'' cannot be implemented
algorithmically, since $s$ and $\ell$ are not known until the
classification is completed.

Instead of up-from-zero situations in the lattice path~$W_n$ of
Part~\ref{sec:lattice-paths} (see Corollary~\ref{cor:exp-up-from-zero}), we
have to consider a
\emph{down-to-zero situation}. This
is a point $(t, 0) \in W_n$ such that
$(t-1, 1) \in W_n$. For a randomly (as described in
Section~\ref{sec:description}) chosen path of length~$n$, we have
\begin{multline*}
  \E{\text{number of down-to-zero situations on $W_n$}} =
  \frac12\left(\E{Z_n} - 1\right)
  = \frac12 \bigl(\Hodd_{n+1} - 1\bigr) \\
  = \frac14 \log n + \frac{\gamma+\log2-2}{4}
  + \frac{\iverson*{$n$ even} + 1}{4n}
  - \frac{9\iverson*{$n$ even} + 2}{24n^2}
  + \frac{\iverson*{$n$ even}}{2n^3}
  + \Oh[Big]{\frac{1}{n^4}},
\end{multline*}
asymptotically as $n$ tends to infinity.

\begin{lemma}\label{lem:expected:comparisons:clairvoyant}
Let $n\ge 2$.
  The expected partitioning cost of strategy ``Clairvoyant''
  is
  \begin{align*}
    \E{\Pcv_n} &= \frac32 n - \frac12 \Hodd_n -\frac{13}8 + \frac{3\iverson*{$n$ odd}}{8n} + \frac{\iverson*{$n$ even}}{8(n-1)}\\
    &=\frac 32n - \frac14\log n - \frac14\gamma - \frac14\log2 - \frac{13}8 +
      \frac 1{8n} + \frac1{12n^2} + O\Bigl(\frac1{n^{3}}\Bigr).
  \end{align*}
  The corresponding generating function is
  \begin{multline*}
    \Pcv(z)=\sum_{n\ge 2} \E{\Pcv_n} z^n 
    = \frac{3}{2(1-z)^2} -\frac{\artanh(z)}{2(1-z)} \\
    -\frac{25z^2}{8(1-z)} +\frac{3+z}{8}\artanh(z) -\frac32 -\frac{23z}{8}.
  \end{multline*}
\end{lemma}

\begin{theorem}\label{thm:clairvoyant:cost}
  For $n\ge 4$, the average number of comparisons in the dual-pivot
  quicksort algorithm ``Clairvoyant'' (with oracle)
  when sorting a list of $n$ elements is
  \begin{multline*}
    \E{\Ccv_n} = 
    \frac{9}{5}nH_n + \frac{1}{5}n\Halt_n -\frac{89}{25}n + \frac{77}{40}H_n+\frac{3}{40}\Halt_n+\frac{67}{800}-\frac{(-1)^n}{10}\\+
     \frac{\iverson*{$n$ even}}{320}\Bigl(\frac{1}{n-3}+\frac{3}{n-1}\Bigr)-\frac{\iverson*{$n$ odd}}{320}\Bigl(\frac{3}{n-2}+\frac{1}{n}\Bigr).
   \end{multline*}
\end{theorem}

Again, the asymptotic behavior follows from the exact result.

\begin{corollary}\label{cor:clairvoyant:cost:asy}
  The average number of comparisons in the dual-pivot
  quicksort algorithm ``Clairvoyant'' (with oracle)
  when sorting a list of $n$ elements is
  \begin{equation*}
    \E{\Ccv_n} = 
    \frac95n \log n + A n + B \log n + C
    + \frac{D}{n} + \frac{E}{n^2} + \frac{F\iverson*{$n$ even} + G}{n^3}
    + \Oh[Big]{\frac{1}{n^4}}
  \end{equation*}
  with
  \begin{align*}
    A &= \frac95\gamma
    - \frac{1}{5} \log 2
    - \frac{89}{25}
    = -2.6596412392892\dots, &
    B &= \frac{77}{40}
    = 1.925, \\
    C &= \frac{77}{40}\gamma
    - \frac{3}{40}\log 2
    + \frac{787}{800}
    = 2.042904116393455\dots, &
    D &= \frac{13}{16} = 0.8125, \\
    E &= - \frac{77}{480} = -0.160416666666666\dots, \\
    F &= \frac{1}{8} = 0.125, &
    G &= - \frac{19}{400} = -0.0475,
  \end{align*}
  asymptotically as $n$ tends to infinity.
\end{corollary}

  The proof concerning strategy ``Clairvoyant'' is analogous to the proof of
  Theorem~\ref{thm:clairvoyant:cost} and
  Corollary~\ref{cor:count:cost:asy}.
  The corresponding generating function is
  \begin{multline*}
    \Ccv(z) = 
    -2\frac{\log(1-z)}{(1-z)^2}-\frac{2\artanh(z)}{5  (1-z)^2}
    - \frac{44}{25  (1-z)^2} 
    + \frac{\artanh(z)}{4 (1-z)}
    + \frac{279}{160 (1-z)}\\
    -\frac{(1-z)^3}{320}\artanh(z) -\frac{2}{75} z^{3}
    + \frac{123}{1600}  z^{2} - \frac{113}{1600} z + \frac{13}{800}
  \end{multline*}
  in this case.

\algrenewcommand{\algorithmiccomment}[1]{\hskip3em // #1}

\section{Pseudocode of Dual-Pivot Quicksort Algorithms}
\label{sec:pseudocode}
In this supplementary section, we give the full pseudocode for the 
strategies ``Count'' (Algorithm~\ref{algo:count}) and ``Clairvoyant'' (Algorithm~\ref{algo:clairvoyant}) turned into dual-pivot quicksort algorithms.
\renewcommand{\alglinenumber}[1]{\footnotesize{#1} }

\begin{algorithm}
    \caption{Dual-Pivot Quicksort Algorithm ``Count''}\samepage\label{algo:count}
    \textbf{procedure} \textit{Count}($\textit{A}$, $\textit{left}$, $\textit{right}$)
    \medskip
    \begin{algorithmic}[1]
        \If{$\textit{right} \leq \textit{left}$}
            \State \Return
        \EndIf
        \If{$A[\textit{right}] < A[\textit{left}]$}
        \State swap \textit{A}[\textit{left}] and \textit{A}[\textit{right}]
        \EndIf
        \State $\texttt{p} \gets A[\textit{left}]$
        \State $\texttt{q} \gets A[\textit{right}]$
        \State $\texttt{i} \gets \textit{left} + 1$;
               $\texttt{k} \gets \textit{right} - 1$;
               $\texttt{j} \gets \texttt{i}$
        \State $\texttt{d} \gets 0$
        \Comment{$\texttt{d}$ holds the difference of the number of small and large elements.}
        \While{$\texttt{j} \leq \texttt{k}$}
            \If{$\texttt{d} \geq 0$}
                \If{$\textit{A}[\texttt{j}] < \textit{p}$}
                    \State swap $\textit{A}[\texttt{i}]$ and $\textit{A}[\texttt{j}]$
                    \State $\texttt{i} \gets \texttt{i} + 1$;
                           $\texttt{j} \gets \texttt{j} + 1$;
                           $\texttt{d} \gets \texttt{d} + 1$
                \Else
                    \If{$\textit{A}[\texttt{j}] < \textit{q}$}
                        \State $\texttt{j} \gets \texttt{j} + 1$
                    \Else
                        \State swap $\textit{A}[\texttt{j}]$ and $\textit{A}[\texttt{k}]$
                        \State $\texttt{k} \gets \texttt{k} - 1$;
                               $\texttt{d} \gets \texttt{d} - 1$
                    \EndIf
                \EndIf
            \Else
                \If{$\textit{A}[\texttt{k}] > \textit{q}$}
                    \State $\texttt{k} \gets \texttt{k} - 1$;
                           $\texttt{d} \gets \texttt{d} - 1$
                \Else
                    \If{$\textit{A}[\texttt{k}] < \textit{p}$}\\
                                            \hskip3em\Comment{Perform a cyclic rotation to the left, i.\,e.,}\\
                        \hskip3em\Comment{$\texttt{tmp} \gets \textit{A}[\texttt{k}]$;
                                          $\textit{A}[\texttt{k}] \gets \textit{A}[\texttt{j}]$;
                                          $\textit{A}[\texttt{j}] \gets \textit{A}[\texttt{i}]$;
                                          $\textit{A}[\texttt{i}] \gets \texttt{tmp}$}

                        \State \textit{rotate3}($\textit{A}[\texttt{k}], \textit{A}[\texttt{j}],
                                        \textit{A}[\texttt{i}]$)
                        \State $\texttt{i} \gets \texttt{i} + 1$;
                               $\texttt{d} \gets \texttt{d} + 1$
                    \Else
                        \State swap $\textit{A}[\texttt{j}]$ and $\textit{A}[\texttt{k}]$
                    \EndIf
                    \State $\texttt{j} \gets \texttt{j} + 1$
                \EndIf
            \EndIf
        \EndWhile
        \State swap $\textit{A}[\textit{left}]$ and $\textit{A}[\texttt{i}-1]$
        \State swap $\textit{A}[\textit{right}]$ and $\textit{A}[\texttt{k}+1]$
        \State \textit{Count}(\textit{A}, \textit{left}, $\texttt{i} - 2$)
        \State \textit{Count}(\textit{A}, $\texttt{i}$, $\texttt{k}$)
        \State \textit{Count}(\textit{A}, $\texttt{k}$ + 2, \textit{right})
    \end{algorithmic}
\end{algorithm}

\begin{algorithm}
    \caption{Dual-Pivot Quicksort Algorithm ``Clairvoyant''}\samepage\label{algo:clairvoyant}
    \textbf{procedure} \textit{Clairvoyant}($\textit{A}$, $\textit{left}$, $\textit{right}$)
    \medskip
    \begin{algorithmic}[1]
        \If{$\textit{right} \leq \textit{left}$}
            \State \Return
        \EndIf
        \If{$A[\textit{right}] < A[\textit{left}]$}
        \State swap \textit{A}[\textit{left}] and \textit{A}[\textit{right}]
        \EndIf
        \State $\texttt{p} \gets A[\textit{left}]$
        \State $\texttt{q} \gets A[\textit{right}]$
        \State $\texttt{i} \gets \textit{left} + 1$;
               $\texttt{k} \gets \textit{right} - 1$;
               $\texttt{j} \gets \texttt{i}$
        \State $\texttt{d} \gets \#(\text{small elements}) - \#(\text{large elements})$
        \Comment{$\texttt{d}$ is given by an oracle.}
        \While{$\texttt{j} \leq \texttt{k}$}
            \If{$\texttt{d} \geq 0$}
                \If{$\textit{A}[\texttt{j}] < \textit{p}$}
                    \State swap $\textit{A}[\texttt{i}]$ and $\textit{A}[\texttt{j}]$
                    \State $\texttt{i} \gets \texttt{i} + 1$;
                           $\texttt{j} \gets \texttt{j} + 1$;
                           $\texttt{d} \gets \texttt{d} - 1$
                \Else
                    \If{$\textit{A}[\texttt{j}] < \textit{q}$}
                        \State $\texttt{j} \gets \texttt{j} + 1$
                    \Else
                        \State swap $\textit{A}[\texttt{j}]$ and $\textit{A}[\texttt{k}]$
                        \State $\texttt{k} \gets \texttt{k} - 1$;
                               $\texttt{d} \gets \texttt{d} + 1$
                    \EndIf
                \EndIf
            \Else
                \If{$\textit{A}[\texttt{k}] > \textit{q}$}
                    \State $\texttt{k} \gets \texttt{k} - 1$;
                           $\texttt{d} \gets \texttt{d} + 1$
                \Else
                    \If{$\textit{A}[\texttt{k}] < \textit{p}$}
                        \State \textit{rotate3}($\textit{A}[\texttt{k}], \textit{A}[\texttt{j}],
                                        \textit{A}[\texttt{i}]$)
                        \State $\texttt{i} \gets \texttt{i} + 1$;
                               $\texttt{d} \gets \texttt{d} - 1$
                    \Else
                        \State swap $\textit{A}[\texttt{j}]$ and $\textit{A}[\texttt{k}]$
                    \EndIf
                    \State $\texttt{j} \gets \texttt{j} + 1$
                \EndIf
            \EndIf
        \EndWhile
        \State swap $\textit{A}[\textit{left}]$ and $\textit{A}[\texttt{i}-1]$
        \State swap $\textit{A}[\textit{right}]$ and $\textit{A}[\texttt{k}+1]$
        \State \textit{Clairvoyant}(\textit{A}, \textit{left}, $\texttt{i} - 2$)
        \State \textit{Clairvoyant}(\textit{A}, $\texttt{i}$, $\texttt{k}$)
        \State \textit{Clairvoyant}(\textit{A}, $\texttt{k}$ + 2, \textit{right})
    \end{algorithmic}
\end{algorithm}


\end{document}

%% file: figure-decomp-leavings.tex
  \begin{tikzpicture}[xscale=-0.25, yscale=0.25, latticepath/.style={very thick}]

    \draw (24,-5) -- (24,14);
    \draw (-24,0) -- (25,0);
    \draw[dotted] (-22,0) -- (-22,10);

    \newcommand{\Cpath}[3][]{
      \draw[latticepath, #1] ($(0,0) + #2$) -- ($(4,0) + #2$) --
      ($(2,2.82842712474619) + #2$) -- cycle;
      \node at ($(2,0) + #2$) [above] {#3};
    }
    \newcommand{\Cpathmirr}[3][]{
      \draw[latticepath, #1] ($(0,0) + #2$) -- ($(4,0) + #2$) --
      ($(2,-2.82842712474619) + #2$) -- cycle;
      \node at ($(2,0) + #2$) [below] {#3};      
    }

    \newcommand{\Cpathdown}[2]{
      \Cpath{#1}{#2}
      \draw[latticepath] ($(4,0) + #1$) -- ($(5,-1) + #1$);
    }

    \node at (-22,10) [right] {$d$};
    \Cpathdown{(-22,10)}{$\CDyckpath$}
    \Cpathdown{(-17,9)}{$\CDyckpath$}

    \draw[dotted] (-11,7) -- (-7,3);
    \draw[latticepath] (-6,2) -- (-5,1);
    \Cpathdown{(-5,1)}{$\CDyckpath$}

    \node[circle,inner sep=1.5pt,fill] at (0,0) {};
    \node[] at (0,0) [below] {$\CZero$};
    \draw[latticepath] (0,0) -- (1,1);
    \Cpath{(1,1)}{$\CDyckpath$}
    \draw[latticepath] (5,1) -- (6,0);
    \draw[latticepath] (0,0) -- (1,-1);    
    \Cpathmirr{(1,-1)}{$\CReflectedDyckpath$}
    \draw[latticepath] (5,-1) -- (6,0);

    \node[circle,inner sep=1.5pt,fill] at (6,0) {};
    \node[] at (6,0) [below] {$\CZero$};
    \draw[latticepath] (6,0) -- (7,1);
    \Cpath{(7,1)}{$\CDyckpath$}
    \draw[latticepath] (11,1) -- (12,0);
    \draw[latticepath] (6,0) -- (7,-1);    
    \Cpathmirr{(7,-1)}{$\CReflectedDyckpath$}
    \draw[latticepath] (11,-1) -- (12,0);

    \node[circle,inner sep=1.5pt,fill] at (12,0) {};
    \node[] at (12,0) [below] {$\CZero$};
    \draw[dotted] (12,2.414) -- (18,2.414);
    \draw[dotted] (12,-2.414) -- (18,-2.414);

    \node[circle,inner sep=1.5pt,fill] at (18,0) {};
    \node[] at (18,0) [below] {$\CZero$};
    \draw[latticepath] (18,0) -- (19,1);
    \Cpath{(19,1)}{$\CDyckpath$}
    \draw[latticepath] (23,1) -- (24,0);
    \draw[latticepath] (18,0) -- (19,-1);    
    \Cpathmirr{(19,-1)}{$\CReflectedDyckpath$}
    \draw[latticepath] (23,-1) -- (24,0);

    \node at (-22,0) [below] {$n$};

  \end{tikzpicture}
